\documentclass[12pt, reqno]{amsart}

\usepackage[english]{babel}
\usepackage{amsthm} 
\usepackage{amssymb}
\usepackage{amsmath}
\usepackage{hyperref}
\usepackage{amsfonts}
\usepackage{stmaryrd}
\usepackage{fancyhdr}
\usepackage[nocompress]{cite}
\usepackage{graphicx,color}
\usepackage{epstopdf}
\usepackage[left=1.1in, right=1.1in, top=1.1in, bottom=1.1in]{geometry}
\usepackage{mathrsfs}  
\usepackage{subfig}

\newtheorem{theorem}{Theorem}[section]
\newtheorem{prop}[theorem]{Proposition}

\newtheorem{hypothesis}[theorem]{Hypothesis}
\newtheorem{lemma}[theorem]{Lemma}
\newtheorem{definition}[theorem]{Definition}
\newtheorem{remark}[theorem]{Note}

\newtheorem{corollary}[theorem]{Corollary}

\newtheorem{example}[theorem]{Example}

\newcommand{\A}{\mathcal{A}}

\newcommand{\N}{\mathbb{N}}
\newcommand{\E}{\mathbb{E}}
\newcommand{\cL}{\mathcal{L}}

\newcommand{\cP}{\mathcal{P}}

\newcommand{\hf}{\hfill$\Box$} 
\newcommand{\pa}{\partial}
\newcommand{\lv}{\left\vert}
\newcommand{\rv}{\right\vert}
\newcommand{\be}{\begin{equation}}
\newcommand{\ee}{\end{equation}}
\newcommand{\dd}{,{\dots},}

\newcommand{\CVone}{D_{V_1}(\R)}
\newcommand{\CV}{D_{V}(\R^N)}

\newcommand{\Xx}{X_t^{(x)}}
\newcommand{\J}{J}

\newcommand{\R}{\mathbb{R}}

\usepackage{booktabs} 
\usepackage{array} 
\usepackage{paralist} 
\usepackage{verbatim} 
\usepackage{fancyvrb}
\usepackage{float}
\usepackage{caption}
\usepackage{bbm} 

\usepackage{hyperref}
\usepackage{url}
\usepackage[toc,page]{appendix}
\usepackage{color}

\usepackage[shortlabels]{enumitem} 

\newcommand{\jacobian}[2]{\frac{\partial}{\partial #2} #1}  

\newcommand{\DV}{\Xi}

\begin{document}
\title[]{Uniform in time estimates  for the weak error of the Euler method for SDEs and a Pathwise Approach to Derivative Estimates for Diffusion Semigroups }
\author{D. Crisan, P. Dobson and M. Ottobre}

\address{\noindent \textsc{Dan Crisan, Department of Mathematics, Imperial College London, Huxley Building, 180 Queen's Gate,  London SW7 2AZ, UK}} 
\email{d.crisan@imperial.ac.uk}
\address{\noindent \textsc{Paul Dobson, Department of Mathematics, Heriot-Watt University, Edinburgh EH14 4AS, UK}} 
\email{pd14@hw.ac.uk}
\address{\noindent \textsc{Michela Ottobre, Department of Mathematics, Heriot-Watt University, Edinburgh EH14 4AS, UK}} 
\email{m.ottobre@hw.ac.uk}

\begin{abstract}  We present a criterion for uniform in time convergence of the weak error of the  Euler scheme for Stochastic Differential equations (SDEs). The criterion requires i) exponential decay in time of the space-derivatives of the semigroup associated with the SDE and  ii) bounds on (some) moments of the Euler approximation.  We show by means of examples (and counterexamples) how both i) and ii) are needed to obtain the desired result. If the weak error converges to zero uniformly in time, then convergence of ergodic averages follows as well. 
 We also show that Lyapunov-type conditions are neither sufficient nor necessary in order for the weak error of the Euler approximation to converge uniformly in time and clarify relations between the validity of Lyapunov conditions, i) and ii). 

Conditions for ii) to hold are studied in the literature.   Here we produce sufficient conditions for i) to hold.  The study of derivative estimates has attracted a lot of attention, however not many results are known in order to guarantee exponentially fast decay of the derivatives.   Exponential decay of derivatives  typically follows from coercive-type conditions involving the vector fields appearing in the equation and their commutators; here we focus on the case in which  such coercive-type conditions are non-uniform in space. To the best of our knowledge, this situation is unexplored in the literature, at least on a systematic level. To obtain results under such space-inhomogeneous conditions we initiate a pathwise approach to the study of derivative estimates  for diffusion semigroups and combine this pathwise method with the use of Large Deviation Principles. 
 \vspace{5pt}

\vspace{5pt}
\\
{\sc AMS Classification (MSC 2010).} 65C20, 65C30, 60H10 , 65G99, 47D07, 60J60. 

\end{abstract}

\date{\today}

\maketitle

%
%
%
%
%


\section{Introduction}
 We consider  stochastic differential equations (SDEs) in $\R^N$ of the form \be\label{SDE}
X_t=X_0 + \int_0^t V_0(X_s) ds + \sqrt{2} \sum_{i=1}^d\int_0^t  V_i(X_s) \circ  dB^i (s), \quad X_0=x,
\ee
where $V_0,\ldots,V_d$ are smooth  vector fields on $\R^N$, $\circ$ denotes Stratonovich integration and  $B^1(t) \dd B^d(t)$  are one dimensional independent standard Brownian motions.   In the first part of this paper we will be concerned with the study of  numerical approximations for  SDEs of the form \eqref{SDE}; in particular we will produce criteria in order for  the (explicit) Euler approximation of  the SDE \eqref{SDE}  to have   weak error which converges  uniformly in time.
To make these criteria easy to use in practice, in the second part of the paper we produce results  which, while of independent interest,  can be employed to check when such criteria are satisfied.

Let us  explain the main results of this paper in more detail. Let $X_t$ be the solution of \eqref{SDE},  $\{Y_{t_n}^{\delta}\}_{n \in \N}$ the corresponding Euler approximation with time-step $\delta$ (see \eqref{eq:eulerdiscrete}) and $\{Y_t^{\delta}\}_{t \geq 0}$ a continuous-time interpolant of $\{Y_{t_n}^{\delta}\}$ (see \eqref{eq:continterp}). Weak error bounds  typically studied in the literature are of the form
\begin{equation}
\sup_{t \in [0,T]}\lvert\mathbb{E}[\varphi(X_t)] - \mathbb{E}[\varphi(Y_t^\delta)]\rvert \leq K(T) \delta\, ,
\end{equation}
where $\varphi$ is a  sufficiently smooth and bounded function, 
see for example \cite[Section 9.7]{KloedenPlaten}. 
In the simplest case  $K(T)$ is of the form $K(T)=ce^{cT}$, for some constant $c>0$. 
Here we study sufficient conditions in order to guarantee the validity of weak error bounds which are uniform in time, i.e. of the form
\begin{equation}\label{stronguniform}
\sup_{t \geq 0}\lvert\mathbb{E}[\varphi(X_t)] - \mathbb{E}[\varphi(Y_t^\delta)]\rvert \leq K \delta, 
\end{equation}
where, crucially, $K$ is {\em independent of time} (although it will depend on $\varphi$ and on the coefficients of the equation).  
Clearly, bounds of the form \eqref{stronguniform} cannot hold in general. Whether they hold or not will depend on both the SDE and the chosen numerical method. As already mentioned, in this paper we consider  the Euler method, but the approach we take is general and can be extended to a wider class of methods. Note that if \eqref{stronguniform} holds,   one does not need to adapt the time- step during the simulation to keep a given threshold accuracy. So this line of research is in a different spirit to adaptive time-stepping methods such as those introduced in \cite{Lamberton, Lord}.  

 Sufficient conditions in order for estimates of the type \eqref{stronguniform} to hold are contained in Section \ref{sec:Eulerestimate}, see  Theorem \ref{thm:globalerror}. To explain the content of  such a  theorem, let us briefly recall the definition of the Markov Semigroup $\{\cP_t\}_{t\geq 0}$  acting  on the space of bounded and measurable functions $f:\R^N\to \R$  and associated to the SDE \eqref{SDE}, namely
\begin{equation}\label{eq:semigpdef}
    \cP_tf(x) := \mathbb{E}[f(X_t) \vert X_0=x], \quad x\in \R^N.
\end{equation}
Theorem \ref{thm:globalerror} may then be informally stated as follows: suppose the SDE \eqref{SDE} is elliptic and the coefficients $V_0, V_1, \dd ,V_d$ grow at most polynomially; if

\smallskip
{\bf i)} the space-derivatives of the semigroup $\cP_t$ decay exponentially fast in time (precise statement of this assumption is in Hypothesis \ref{hyp:euler} \ref{ass:gradest}) and

\smallskip
{\bf ii)} some moments of the Euler approximation $\{Y_{t_n}^\delta\}$ of $X_t$ are uniformly bounded in time (see Hypothesis \ref{hyp:euler} \ref{hypmom}), 

\medskip
\noindent
then  \eqref{stronguniform} holds.  Note that while Theorem \ref{thm:globalerror} assumes that the noise in \eqref{SDE} is non-degenerate, see Hypothesis \ref{hyp:euler} \ref{ass:ellipiticity}, we believe that our result is stable to relaxing this assumption and this will be the subject of future work. Indeed, some of the examples that we exhibit cover the degenerate noise case as well.  

Sufficient conditions in order for ii) to hold are discussed for example in \cite{Mattingly, Talay, TalayTubaro}, we will be more precise on this point in  Note \ref{note:reducedassumptions}. 
So in this paper we focus on criteria in order for i) to hold. We moreover give examples to show that i) $\not\Rightarrow$ ii) (Example \ref{ex:Grusin}),  ii) $\not \Rightarrow$ i) (Example \ref{ex:circlewithconfinement}) and neither i) nor ii) by themselves imply \eqref{stronguniform}, i.e. i) $\not\Rightarrow$ \eqref{stronguniform} (Example \ref{ex:Grusin}) and ii) $\not\Rightarrow$ \eqref{stronguniform} (Example \ref{ex:circlewithconfinement}). Furthermore, because the uniform convergence \eqref{stronguniform} implies convergence of the ergodic averages, our criterion gives also a sufficient condition for the latter convergence to hold, see Corollary \ref{cor:ergodicaverages}. 

We also discuss the relation between i), ii) and some Lyapunov-type conditions; we do this in detail in Note \ref{note:Lyapunov} and there we will be also more precise about the relation between our results and results based on Lyapunov conditions that can be found in the literature. For the time being let us just notice that in this paper we provide examples to show that Lyapunov conditions are not sufficient in order for \eqref{stronguniform} to hold -- we do this   both in the case in which the noise in the SDE is  degenerate (see Example \ref{ex:circlewithconfinement}) and when it is non-degenerate (see Example \ref{ex:xcubed}). As proven in \cite[Section 3]{TalayTubaro},  under some assumptions on the coefficients of the SDE,   Lyapunov conditions (for example of the type \eqref{eq:Lyapcond}) are sufficient to obtain ii) (i.e. boundedness  of some moments of the Euler approximation); however, as we have already said, they are not sufficient to obtain \eqref{stronguniform}. Viceversa, Lyapunov conditions are also not  necessary in order to obtain uniform approximations, see Note \ref{note:Lyapunov} for clarifications.   

Let us now comment more on point i). Assuming that  $V$ is some direction\footnote{More precisely, $V$ is a vector field  on $\R^N$ and, as we will recall in Section \ref{sec2}, there exists a canonical identification between vector fields and first order differential operators, see \eqref{ID}.} along which the semigroup $\cP_t$ is differentiable (so that the LHS of \eqref{expdecayintro} below makes sense),  we will give conditions in order for estimates of the following type to hold
\begin{equation}\label{expdecayintro}
\lvert V\cP_tf(x) \rvert  \leq  c  \, u(x) e^{-\lambda_0 t}\, ,  \quad  \forall  t>0\,, 
\end{equation}
 for some constant $c>0$ (which depends on $f$), see Theorem \ref{prop:OAC} and Theorem \ref{prop:LOAChd}, and for some appropriate positive function $u$ to be introduced in Section \ref{sec:DonskerVaradhan} below.  We note in passing that estimates of the form \eqref{expdecayintro} are more general than those required in Section \ref{sec:Eulerestimate}, see \eqref{eq:expdecayuptoorder4nonuniform}, for Theorem \ref{thm:globalerror} to hold; indeed (because Theorem \ref{thm:globalerror} refers to elliptic SDEs), in that context only derivative estimates in the coordinate directions are needed. The bound \eqref{expdecayintro} is more general in the sense that $V$ can be any direction and we will further clarify the relation between \eqref{expdecayintro} and \eqref{eq:expdecayuptoorder4nonuniform} in Note  \ref{Note:ellipticity}.

The study of derivative estimates for Markov semigroups has a long history and it has been tackled by using various approaches,  see e.g. \cite{Lunardi, Priola, Dragoni, Bakry, mythesis} and references therein.  As is well-known, without any quantitative conditions on the vector fields appearing in \eqref{SDE} (i.e. if only ellipticity/hypoellipticity or other regularity assumptions are made),  only the following smoothing-type estimates hold
$$
\lvert V\cP_tf(x) \rvert  \leq  u(x) \frac{1}{t^{\gamma}}\, ,  \quad  \mbox{for }  t \in (0,1) \, ,
$$
where $\gamma>0$ is an appropriate exponent which depends on the direction $V$, and $f$ is continuous and bounded,  see \cite{KusStrder, mythesis, MV_I, Bakry, Nee, Lunardi},  and most of the literature is devoted to estimates of the above type.   In \cite{CrisanOttobre} the authors introduced a sufficient condition in order for  \eqref{expdecayintro} to hold, the so-called {\em Obtuse Angle Condition} (OAC) (see Appendix \ref{app:UFG} for a precise statement of the results of \cite{CrisanOttobre}); we say that the OAC is satisfied by the vector fields $V$ and $V_0$ (where $V_0$ is the drift of \eqref{SDE}) if 
\begin{equation}\label{OAC1d}
\xi^T ([V,V_0](x))   (V(x))^T \xi  \leq -\lambda \, \lvert  \xi^T V(x)\rvert^2, \quad \mbox{ for every } x, \xi\in\R^N\,,
\end{equation}
where the superscript $^T$ denotes transpose (so that e.g. $\xi^T$ is a row vector). 
 This is a coercivity-type condition and in the above such a coercivity is required to hold uniformly in space in the sense that $\lambda>0$ is a constant independent of $x$. In contrast,  in this paper we discuss the case in which $\lambda$ is allowed to be a continuous function of $x$. That is, we consider the following condition
\begin{equation}\label{LOAC1d}
\xi^T\left([V,V_0](x) \right)  (V(x))^T \xi \leq -\lambda(x) \lvert \xi^T V(x)\rvert^2, \quad \mbox{ for every } x, \xi\in\R^N \, ,
\end{equation}
which we refer to as  the 
{\em Local Obtuse Angle Condition} (LOAC).  In Section \ref{sec:pathwiseOAC} we give a simple example to further explain why we name \eqref{LOAC1d} the {\em local} OAC, see comments after equation \eqref{eq:arctanintro}. 
Under no further assumptions  on the function $\lambda: \R^N \rightarrow \R$ (neither on the regularity nor on the sign of such a function)  we show that the following holds
\begin{equation}\label{gradest}
	\lvert V\cP_tf(x)\rvert \leq c \, \mathbb{E}\left[\exp\left(-2\int_0^t\lambda(X_s)ds\right) \right]^{\frac{1}{2}}, 
	\end{equation}
for some constant $c>0$.  In order to obtain estimates of the form \eqref{gradest} under the local condition \eqref{LOAC1d},  we need to gain  detailed control over  the paths of the diffusion $X_t$; for this reason we initiate in this paper a pathwise version of the Bakry-Emery approach \cite{Bakry} to the study of derivative estimates for Markov semigroups. This is the content of Section \ref{sec:pathwiseOAC}.  
Clearly, if  $\lambda(x) \geq \lambda_0>0$ for some constant $\lambda_0$ then \eqref{expdecayintro} (with $u(x)$ constant) follows from \eqref{gradest}. If $\lambda(x)>0$ is just positive, i.e. if it is not uniformly  bounded from below by a positive constant, or even negative for some $x \in \R^N$, one can still obtain \eqref{expdecayintro} from \eqref{gradest}. This is what we show in Section  \ref{sec:DonskerVaradhan}. Roughly speaking, in Section \ref{sec:DonskerVaradhan} we show that if there exists a set $F$ such that $\lambda(x)\geq \lambda_0>0$ for every $x\in F$ and the processs spends enough time in such a set, then one can still obtain \eqref{expdecayintro} from \eqref{gradest}. In order to obtain such results we make use of Large Deviation principles; in particular, we use (and generalise) some estimates on  functionals of the  {occupation measure}  which have been obtained by Donsker and Varadhan in \cite{Donsker1}-\cite{Donsker4}. 
This provides a link between the study of derivative estimates for Markov Semigroups and Large Deviations theory and allows one to give an explicit characterization of the function $u(x)$ appearing in \eqref{expdecayintro}.

This paper is organised as follows. In Section \ref{sec2} we set out the  standing notation for the rest of the paper; Section \ref{sec:Eulerestimate} contains the main criterion, Theorem \ref{thm:globalerror},  in order for uniform in time bounds \eqref{stronguniform} on the weak error of the Euler scheme to hold. Section \ref{sec:pathwiseOAC}  presents the pathwise approach developed to obtain estimates of the type \eqref{gradest} from the non-uniform coercivity  condition \eqref{LOAC1d}. This pathwise approach is, to the best of our knowledge, new and inspired by the Bakry-Emery approach \cite{Bakry};  we explain in Note \ref{compBakry} the reason why, under non-uniform coercivity  conditions, classical Bakry-type semigroup techniques can no longer be used. To present the main ideas without cumbersome notations, all the results of Section \ref{sec:pathwiseOAC} are presented in one dimension first, and then extended to SDEs in $\R^N$; in the latter case we impose some extra assumptions on the commutators between the vector fields appearing in the SDE \eqref{SDE} - in short, we assume a commutator structure which is similar to the one assumed in the Hypocoercivity Theory \cite{Villani}. Full extensions  to $\R^N$ (i.e. extensions that require less assumptions on the commutator structure) are lengthy and significantly more technical, and they will be the tackled in  \cite{Paul}.  In Section \ref{sec:DonskerVaradhan} we explain how to obtain exponential decay estimates of the form \eqref{expdecayintro} once estimates of the type \eqref{gradest} have been derived by using the results of Section \ref{sec:pathwiseOAC}.  The results of Section \ref{sec:DonskerVaradhan} are completely dimension-independent, so they are presented straight away in $\R^N$. Note that while the results of  Section \ref{sec:Eulerestimate} hold for elliptic diffusions, no such ellipticity assumption is enforced in subsequent sections and the results of Section \ref{sec:pathwiseOAC}  and Section \ref{sec:DonskerVaradhan}  hold for any hypoelliptic or even UFG diffusion (for the definition of UFG diffusion please see Appendix \ref{app:UFG}). 
Section \ref{sec6} contains several examples and counterexamples to illustrate cases where the results developed in this paper apply. Finally, Appendix \ref{app:UFG} contains some background notions, for the readers' convenience, while  Appendix B contains auxiliary proofs.

\section{Notation and Preliminaries}\label{sec2}
Given a vector field $V=V(x)$ on $\R^N$,  $V=(V^1(x), V^2(x),$ $..., V^N(x))$ $x \in \R^N$, we refer to the functions $\{V^j(x)\}_{1 \leq j \leq N}$ as the {\em components} or {\em coefficients} of the vector field.   We  say that a vector field is smooth or that it is $C^{\infty}$ if all the components $V^j(x)$, $j=1 \dd N$, are $C^{\infty}$ functions.  As a {\bf standing assumption},  throughout the paper we only consider vector fields which are smooth. We do not repeat this assumption in all the statements. 
We can interpret $V$ both as a vector-valued function on $\R^N$ and as a first order differential operator on $\R^N$, through the canonical identification
\begin{equation}\label{ID}
V=(V^1(x), V^2(x) \dd V^N(x)) \quad \mbox{ or }
\quad V= \sum_{j=1}^NV^j(x)\pa_j, 
 \quad x \in \R^N, \pa_j= \pa_{x^j} \,.
\end{equation}
We will use this identification throughout and we will not use different notations to distinguish the vector field from the differential operator, but will make comments when confusion may arise. 
Throughout this paper we shall denote by $\partial_i V^j$ the $i$-th derivative of the $j$-th component of $V$; if $N=1$ then we will write the first (second,  respectively) derivative of the coefficient as $V'(x)$ ($V''(x)$, respectively). We shall use the notation $V^{(n)}$ to denote the $n$-th order differential operator obtained by iterating $V$ $n$ times, that is
\begin{equation*}
    V^{(n)}f(x) = (\underbrace{V\ldots V}_{n\text{ times}}f)(x).
\end{equation*}
If $f(t,x)$ is a function of time and space, as customary $\pa_t f(t,x)$ and $\pa_{x^i}f(t,x)$, respectively,  denote the time derivative and the derivative in the space-coordinate direction $x^i$, respectively. Given two differential operators $V$ and $W$, the commutator between $V$ and $W$ is defined as
$$
[V,W]:= VW -WV\,, 
$$
and it is a first order differential operator. 
Equivalently, when we view $V$ and $W$ as vector fields we may define the commutator of $V$ and $W$ as
$$
[V,W](x):= \frac{\partial W(x)}{\partial x}V(x) -\frac{\partial V(x)}{\partial x} W(x)\,.
$$
Here $\frac{\partial W(x)}{\partial x}$ ($\frac{\partial V(x)}{\partial x}$, respectively) denotes the Jacobian matrix of $W$ ($V$, respectively), i.e. the $ij$-th entry of the matrix $\frac{\partial W(x)}{\partial x}$ is 
$$
\left(\frac{\partial W(x)}{\partial x}\right)_{ij}:=\partial_j W^i(x) \,.
$$
When considering the SDE \eqref{SDE}, we will often want to emphasize the dependence of the solution on the initial datum; to  this end we will use the notation $X_t^{(x)}$. To be more explicit, we denote by $X_t^{(x)}$ the solution to the following SDE in Stratonovich form, 
\begin{equation}\label{eq:SDE}
dX_t^{(x)} = V_0(X_t^{(x)})dt+\sqrt{2}\sum_{i=1}^d V_i(X_t^{(x)}) \circ dB_t^i,\quad  X_0^{(x)} = x\in\R^N,
\end{equation}
where the drift and diffusion coefficients are smooth and such that there is a pathwise unique strong solution to \eqref{eq:SDE}. We may write \eqref{eq:SDE} in It\^o form as
\begin{equation}\label{eq:SDEito}
dX_t^{(x)} = U_0(X_t^{(x)})dt+\sqrt{2}\sum_{i=1}^d V_i(X_t^{(x)})  dB_t^i, \quad X_0^{(x)} = x\in\R^N,
\end{equation} 
where $U_0$ denotes the drift term in the corresponding It\^o form, i.e.
\begin{equation}\label{eq:Itodrift}
U_0^i(x) = V_0^i(x) + \sum_{k=1}^d\sum_{j=1}^N  V_k^j(x) \partial_jV_k^i(x).
\end{equation}
For the sake of clarity we emphasize again that in \eqref{eq:SDE} and \eqref{eq:SDEito} (as well as in \eqref{SDE}) $B^1(t), \dd , B^d(t)$ are one-dimensional independent Brownian motions. 
We denote by $\cL$ the generator of the SDE, i.e. the second order differential operator defined on suitably smooth functions $f:\R^N\to\R$ as
\begin{equation*}
\cL f(x):=V_0f(x)+\sum_{k=1}^dV_k^{(2)}f(x) = \sum_{i=1}^N U_0^i(x)\partial_if(x) + \sum_{i,j=1}^N V_k^i(x)V_k^j(x) \partial_i\partial_jf(x)\, ,
\end{equation*}
and by  $\cL_{(v)}$ the operator obtained from the one defined above by  ``freezing" the value of the coefficients to $v$; that is, 
\begin{equation}\label{starstar}
(\cL_{(v)}f)(y)= \sum_{i=1}^N U_0^i(v)(\partial_if)(y) + \sum_{i,j=1}^N V_k^i(v)V_k^j(v) (\partial_i\partial_jf)(y).
\end{equation}

We shall also use the following function spaces:
\begin{itemize}
    \item  $C_b(\R^N)$ is  the set of all continuous and bounded functions $f:\R^N\to\R$, endowed with the supremum norm
    $$
    \lVert f\rVert_\infty :=\sup_{x\in\R^N}\lvert f(x)\rvert.
    $$
    \item $C_b^n(\R^N)$ is the space of $n$-times differentiable and bounded functions $f:\R^N\to\R$ with bounded derivatives (of order up to $n$),  endowed with the norm
    \begin{align*}
        \lVert f \rVert_{C_b^n} = \sum \lVert \partial_{1}^{\alpha_1}\ldots\partial_{N}^{\alpha_N} f \rVert_\infty, \quad 
    \end{align*}
where the sum is over indices $\alpha_j$'s such that $\sum_{k=1}^N\alpha_k \leq n$ and $\alpha_j \in \{0,  \ldots , n\}$ for every $j$. 
    The space of all infinitely differentiable functions with bounded derivatives of all orders will be denoted by $C_b^\infty(\R^N)$. 
    \end{itemize}
Finally, unless otherwise stated, all the vectors in $\R^N$ are assumed to be column vectors; so, for any $\xi \in \R^N$, $\xi^T$ is a row vector.      

\section{Uniform in time convergence  of the Euler scheme}\label{sec:Eulerestimate}

Let $\{Y_{t_n}^\delta\}_{n\in\N}$ be the (explicit)  Euler approximation with time-step $\delta$ of the SDE \eqref{eq:SDEito}, that is
\be\label{eq:eulerdiscrete}
Y_{t_{n+1}}^\delta = Y_{t_n}^\delta+U_0(Y_{t_n}^\delta) \delta +\sqrt{2} \sum_{k=1}^dV_k (Y_{t_n}^\delta) \Delta B_{t_n}^k, \quad Y_0^\delta=x\, ,
\ee

where $t_n=n\delta$  and $\Delta B_{t_n}=B_{t_{n+1}}-B_{t_n}$.
Define $\{Y_t^\delta\}_{t\geq 0}$ to be the continuous-time interpolant of $\{Y_{t_n}^\delta\}_{n\in\N}$, i.e.
\begin{align}\label{eq:continterp}
dY_t^\delta &= U_0(Y_{t_{n(t)}}^\delta) dt+ \sum_{k=1}^dV_k(Y_{t_{n(t)}}^\delta) dB_t^k, \quad t_{n(t)}=t_i \mbox{ for } t \in [t_i, t_{i+1}) \, ,\\
Y_0^{\delta} & = x. \nonumber
\end{align}
The Brownian motions appearing in \eqref{eq:eulerdiscrete} and in \eqref{eq:continterp} are the same as the one in \eqref{eq:SDE}. 
Note that the continuous-time process $Y_t^{\delta}$ and the discrete time process $Y_{t_n}^{\delta}$ coincide at the points $t_n$ of the mesh (hence the reason why we denote both of them by $Y_{\cdot}^{\delta}$ without risk of confusion). 

The main result of this section is Theorem \ref{thm:globalerror} which gives sufficient conditions under which the Euler scheme weakly approximates the underlying SDE {\em uniformly in time}. The full set of assumptions under which Theorem \ref{thm:globalerror} holds is Hypothesis \ref{hyp:euler} below.  Immediately after stating Theorem \ref{thm:globalerror}, we make several comments on Hypothesis \ref{hyp:euler} and we give a list of cases under which such assumptions are indeed satisfied (see Note \ref{note:reducedassumptions}, Note \ref{note: applicability} and Corollary \ref{cor:additivenoisesummary}). The requirement \eqref{eq:expdecayuptoorder4nonuniform} on the derivatives of the semigroup is then studied in Sections \ref{sec:pathwiseOAC} and \ref{sec:DonskerVaradhan}.

\begin{hypothesis}\label{hyp:euler}
\begin{enumerate}[(a)]
\item \label{ass:ellipiticity} For every $x\in\R^N$ there is a pathwise unique strong solution $\{X_t^{(x)}\}_{t\geq 0}$  to the SDE \eqref{eq:SDEito} and
 the vector fields $V_1,\ldots,V_d$ satisfy a uniform ellipticity condition, i.e. there exists some $\nu>0$ such that
$$
\sum_{k=1}^d \lvert \xi^T V_k(x) \rvert^2 \geq \nu \lvert \xi\rvert^2, \quad \mbox{ for all } \xi,x\in\R^N.
$$
	\item The vector fields $U_0,V_1, \ldots, V_d$ are smooth; both the vector fields themselves and their first and second order derivatives  have at most polynomial growth. That is, there exist some constants $K_1,K_2,K_3,p,q,q'\geq 0$ such that
\begin{align*}
\lvert U_0(x)\rvert + \sum_{k=1}^d \lvert V_k(x)\rvert &\leq K_1(1+\lvert x\rvert^p) \, ,\\
\sum_{i=1}^N\lvert \partial_iU_0(x)\rvert + \sum_{k=1}^d \sum_{i=1}^N\lvert \partial_iV_k(x)\rvert &\leq K_2(1+\lvert x\rvert^{q}) \, ,\\
\sum_{i=1}^N\sum_{j=1}^N\lvert \partial_{i,j}U_0(x)\rvert + \sum_{k=1}^d \sum_{i=1}^N\sum_{j=1}^N\lvert \partial_{i,j}V_k(x)\rvert &\leq K_3(1+\lvert x\rvert^{q'})\, .
\end{align*}	
	 \label{ass:lipschitz}
\item \label{ass:gradest} There exist a constant $\lambda_0>0$ and a positive function $u:\R^N\to\R$ such that for all $f\in C_b^4(\R^N)$ we have
\begin{equation}\label{eq:expdecayuptoorder4nonuniform}
\sum_{k=1}^4\sum_{i_1,\ldots, i_k=1}^d\lvert \partial_{i_1,\ldots,i_k}\cP_tf(x) \rvert \leq u(x) e^{-\lambda_0  t} \lVert f \rVert_{C^4_b} \, .
\end{equation}
\item \label{hypmom} Let $\gamma:= \max\{p+q, 2p\}$ and $\zeta:= \max\{p+q', p+q, 2q, 2p\}$. The function  $u$ appearing in \eqref{eq:expdecayuptoorder4nonuniform} is such that the following bounds hold: 
\begin{align}
    &K_4:=\sup_{s\geq 0}\mathbb{E}\left[\lvert (1+Y_{t_{n(s)}}^\delta\rvert^p)\left(1+ \lvert Y_s^\delta\rvert^{\gamma}\right)u(Y_s^\delta)\right]<\infty, \label{eq:uboundinexp} \\
    &K_5:=\sup_{s\geq 0}\mathbb{E}\left[(1+\lvert Y_{t_{n(s)}}^\delta\rvert^{2p})\left(1+\lvert Y_s^\delta\rvert^{\zeta}\right)u(Y_s^\delta)\right]<\infty, \label{eq:uboundinexp2}\\
    &K_6:=\sup_{s\geq 0}\mathbb{E}\left[(1+\lvert Y_{t_{n(s)}}^\delta\rvert^{4p})u(Y_s^\delta)\right]<\infty. \label{eq:uboundinexp3}
\end{align}
\end{enumerate}
\end{hypothesis}

\begin{theorem}\label{thm:globalerror}
Let Hypothesis \ref{hyp:euler} hold. Then the weak error of the Euler approximation $\{Y_t^\delta\}_{t\geq 0}$ of the SDE \eqref{eq:SDEito} converges to $0$, uniformly in time, as $\delta \to 0$; that is, there exists some constant $K$ depending only on  $\lambda, K_1,\ldots, K_6,d$ and $N$ (but not on $t$) such that for all $\varphi\in C_b^\infty(\R^N)$ and $\delta>0$ we have
\begin{equation}\label{eq:uniformweakconv}
\sup_{t\geq 0}\lvert\mathbb{E}[\varphi(X_t)] - \mathbb{E}[\varphi(Y_t^\delta)]\rvert \leq K \delta \lVert \varphi \rVert_{C^4_b}.
\end{equation}
\end{theorem}

If \eqref{eq:uniformweakconv} holds then we say that the weak error of the Euler approximation  converges to zero uniformly in time and the convergence is of order $\delta$. Before proving Theorem \ref{thm:globalerror}, we make several comments on the statement of the theorem.   

\begin{remark}\label{note:reducedassumptions}
\textup{Some comments on the above result.}
\begin{itemize}
\item \textup{Conditions are given in 
\cite[Corollary 7.5]{Mattingly}, \cite[Theorem 2]{TalayTubaro},   under which the  moment bounds \eqref{eq:uboundinexp}-\eqref{eq:uboundinexp3} do hold, so in this paper we rather focus on the study of conditions under which \eqref{eq:expdecayuptoorder4nonuniform} holds, see Section \ref{sec:pathwiseOAC} and Section \ref{sec:DonskerVaradhan}. We emphasize that such moment bounds are required to hold for the Euler approximation, not for the SDE itself. And on this matter we recall that the  even if the SDE \eqref{eq:SDE} has  moments of all orders, all of them bounded uniformly  in time,  this does not imply that the moments of the numerical approximation will enjoy the same property, see \cite[Lemma 6.3]{Mattingly}. In Example \ref{ex:additivenoiseeuler} we show how to use \cite[Corollary 7.5]{Mattingly} in our context.}
\item \textup{Hypothesis \ref{hyp:euler} \ref{ass:gradest} does not imply Hypothesis \ref{hyp:euler} \ref{hypmom} and moreover Hypothesis \ref{hyp:euler} \ref{ass:gradest} alone is not sufficient for \eqref{eq:uniformweakconv} to hold; indeed in Example \ref{ex:Grusin} we exhibit a simple two-dimensional SDE for which Hypothesis \ref{hyp:euler} \ref{ass:gradest} does hold but the (fourth)  moments of the corresponding Euler approximation do not satisfy Hypothesis \ref{hyp:euler} \ref{hypmom} and \eqref{eq:uniformweakconv} does not hold. We also note that the SDE of Example \ref{ex:Grusin} does satisfy the OAC (which, as we have recalled in the introduction, implies exponential decay of the derivatives, see Appendix \ref{app:UFG} for details); however, as Example \ref{ex:Grusin} shows, the OAC implies neither tightness of the process itself nor of its Euler approximation - in particular it does not imply the bounds \eqref{eq:uboundinexp}-\eqref{eq:uboundinexp3} of Hypothesis \ref{hyp:euler} \ref{hypmom}.}
\item \textup{Hypothesis \ref{hyp:euler} \ref{hypmom} does not imply Hypothesis \ref{ass:gradest} and Hypothesis \ref{hyp:euler} \ref{hypmom} alone is not sufficient to conclude \eqref{eq:uniformweakconv}, see Example \ref{ex:circlewithconfinement}.}
\item \textup{Note that whether Hypothesis \ref{hyp:euler} \ref{hypmom} holds or not may depend on the initial datum $x$ of the SDE \eqref{eq:SDE} and on the chosen step size, see Note \ref{note:Lyapunov} on this. Comments on the relation between Hypothesis \ref{hyp:euler} and Lyapunov-type conditions can also be found in Note \ref{note:Lyapunov}.}
\hf
\end{itemize}
\end{remark}

\begin{remark}\label{note: applicability} \textup{Here we point out some cases in which  Hypothesis \ref{hyp:euler} simplifies. }
\begin{itemize}
\item\textup{If the vector fields $U_0,V_1,\ldots, V_d$ are bounded and have bounded derivatives of all orders then $p=q=q'=0$ in Hypothesis \ref{hyp:euler} \ref{ass:lipschitz}, and   Hypothesis \ref{hyp:euler} \ref{hypmom} simplifies to requiring that \eqref{eq:expdecayuptoorder4nonuniform} holds for a function $u$  such that}
\begin{equation}\label{eq:uboundassumption}
\sup_{t\geq 0}\mathbb{E}\left[u(Y_t^\delta)\right]<\infty. \end{equation}
\textup{For concrete examples that fall within this case and for which the function $u(x)$ can be explicitly constructed see Corollary \ref{cor:additivenoisesummary} and Example \ref{ex:arctan}.}
    \item \textup{Assume the SDE \eqref{eq:SDEito} is elliptic and has a unique strong solution.   If  the vector fields $U_0, V_1, \dd V_d$ grow at most linearly and have bounded derivatives of all orders then Hypothesis \ref{hyp:euler} \ref{ass:lipschitz} holds with $p=1, q=0, q'=0$. In \cite{CrisanOttobre} and \cite{CCDO} it is shown that if the OAC is satisfied by appropriate vector fields (see \cite[Section 3.1]{CCDO} and Appendix \ref{app:UFG} for details of precise statement) 
then the bound \eqref{eq:expdecayuptoorder4nonuniform} holds  with $u(x) =constant$.  Therefore in this case  checking that  Hypothesis \ref{hyp:euler} \ref{hypmom} holds reduces to  verifying the following:
\be\label{eq:bdd4moment}
\sup_{t\geq 0}\mathbb{E}[\lvert Y_t^{\delta}\rvert^{4}]<\infty, 
\ee}
\textup{Example \ref{ex:additivenoiseeuler} gives a class of SDEs that fall within this case.}
\end{itemize}
\end{remark}

Before proving Theorem \ref{thm:globalerror} we state and prove the following standard auxiliary lemma. 
\begin{lemma}\label{lem:ItoforY}
    If $f\in C^2(\R\times\R^N)$ then, for any $t\geq 0$ we have
    \begin{align*}
    f(t,Y_t^\delta)& = f(0,Y_0^\delta) + \int_0^t \left[\partial_sf(s,Y_s^\delta)+ \cL_{(Y_{t_{n(s)}}^\delta)}f(s,Y_s^\delta)\right] ds \\
    &+ \sqrt{2}\sum_{k=1}^d\sum_{i=1}^N \int_0^t V_k^i(Y_{t_{n(s)}}^\delta)\partial_if(s,Y_{s}^\delta) dB_s^k.
    \end{align*}
\end{lemma}

\begin{proof}[Proof of Lemma \ref{lem:ItoforY}]
    Fix $f\in C^2(\R\times\R^N)$ and $t\geq0$ and let $n\in \N$ to be such that $t\in [t_n,t_{n+1})$. Then, using that $Y_t^{\delta}$ solves the SDE \eqref{eq:continterp}, by It\^o's formula we have
    \begin{align*}
    f(t,Y_t^\delta) &= f(t_n,Y_{t_n}^\delta) + \int_{t_n}^t(\partial_sf)(s,Y_s^\delta)+ \sum_{i=1}^N U_0^i(Y_{t_{n(s)}}^\delta)(\partial_if)(s,Y_s^\delta) ds\\
    &+ \int_{t_n}^t\sum_{i,j=1}^N V_k^i(Y_{t_{n(s)}}^\delta)V_k^j(Y_{t_{n(s)}}^\delta) (\partial_i\partial_jf)(s,Y_s^\delta) ds + \sqrt{2}\sum_{k=1}^d\sum_{i=1}^d \int_{t_n}^t V_k^i(Y_{t_{n(s)}}^\delta) \partial_if(s,Y_s^\delta) dB_t^k\\
    &=f(t_n,Y_{t_n}^\delta) + \int_{t_n}^t (\partial_sf)(s,Y_s^\delta)+ \cL_{(Y_{t_{n(s)}}^\delta)}f(s,Y_s^\delta) ds \\
    &+ \sqrt{2}\sum_{k=1}^d\sum_{i=1}^N \int_{t_n}^t V_k^i(Y_{t_{n(s)}}^\delta) \partial_if(s,Y_s^\delta) dB_s^k,
    \end{align*}    
    where in the second equality we have used \eqref{starstar}. Then by using  telescoping sums we have
    \begin{align*}
        f(t,Y_t^\delta) &= f(t,Y_t^\delta)-f(t_n,Y_{t_n}^\delta) + \sum_{m=0}^{n-1}[f(t_{m+1},Y_{t_{m+1}}^\delta) -f(t_m,Y_{t_m}^\delta)]+f(0,Y_0^\delta)\\
        &= f(0,Y_0^\delta) + \int_0^t (\partial_sf)(s,Y_s^\delta)+ \cL_{(Y_{t_{n(s)}}^\delta)}f(s,Y_s^\delta) ds \\
        &+ \sqrt{2}\sum_{k=1}^d\sum_{i=1}^N \int_0^t V_k^i(Y_{t_{n(s)}}^\delta)\partial_if(s,Y_{s}^\delta) dB_s^k
    \end{align*}
\end{proof}
 \begin{proof}[Proof of Theorem \ref{thm:globalerror}]
 Fix $\delta>0, t>0, \varphi\in C_b^\infty(\R^N)$. By applying It\^o's formula in the variable $s$ to $(\cP_{t-s}\varphi)(X_s^{(x)})$ (where the semigroup $\cP_t$ has been  introduced in \eqref{eq:semigpdef}), we have
\begin{align*}
(\cP_{t-s}\varphi)(X_s) & = (\cP_t\varphi)(x)+\int_0^s \pa_r (\cP_{t-r} \varphi) (X_r) dr 
\nonumber \\
& + \int_0^s \cL (\cP_{t-r} \varphi) (X_r) dr + 
\sqrt{2}\sum_{k=1}^d\sum_{i=1}^N\int_0^s V_k^i(X_r)\partial_i(\cP_{t-r} \varphi) (X_r) dB_r^k \,.
\end{align*}
Because $\pa_r (\cP_{t-r} \varphi) (X_r)= -\cL (\cP_{t-r} \varphi) (X_r)$, one gets
\be\label{itoforXsemigp} 
(\cP_{t-s}\varphi)(X_s)  = (\cP_t\varphi)(x)+ \sqrt{2}\sum_{k=1}^d\sum_{i=1}^N\int_0^s V_k^i(X_r)\partial_i(\cP_{t-r} \varphi) (X_r) dB_r^k \,.
\ee
On the other hand, by applying Lemma \ref{lem:ItoforY} with $f(s,y)=\cP_{t-s}\varphi(y)$, we get 
\begin{align}\label{itoforY}
(\cP_{t-s}\varphi)(Y_s^\delta) & = (\cP_{t}\varphi)(x)+ 
\sum_{k=1}^d\int_0^s V_k(Y_{t_{n(r)}}^\delta)\partial_i(\cP_{t-r}\varphi) (Y_{r}^\delta) dB_r^k  \nonumber\\
& + \int_0^s \cL_{(Y_{t_{n(r)}}^\delta)} (\cP_{t-r}\varphi) (Y_{r}^\delta) dr
- \int_0^s \cL_{(Y_r^\delta)} (\cP_{t-r}\varphi)(Y_r^\delta) dr.
\end{align}
Evaluating \eqref{itoforXsemigp} and \eqref{itoforY} at $s=t$, taking expectation and then the difference between the two equations, we obtain
\begin{align}\label{eq:diffinexp}
\E \varphi (X_t) - \E \varphi (Y_t^\delta) & =
\E \int_0^t \left(  \cL_{(Y_{r}^\delta)} - \cL_{(Y_{t_{n(r)}}^\delta)}\right) (P_{t-r}\varphi) (Y_r^\delta) dr .
\end{align}
We can now decompose \eqref{eq:diffinexp} as follows:
\begin{align}\label{eq:decompoferror}
\E \varphi (X_t) - \E \varphi (Y_t^\delta) & = I_1+I_2,
\end{align}
where 
\begin{align*}
    I_1&=\E \int_0^t \left(  \left(\cL_{(Y_{r}^\delta)}\cP_{t-r}\varphi\right)(Y_r^\delta)- \left(\cL_{(Y_{t_{n(r)}}^\delta)}\cP_{t-r}\varphi\right) (Y_{t_{n(r)}}^\delta) \right)dr,\\
    I_2&=\E \int_0^t \left(  \left(\cL_{(Y_{t_{n(r)}}^\delta)}\cP_{t-r}\varphi\right)(Y_{t_{n(r)}}^\delta)- \left(\cL_{(Y_{t_{n(r)}}^\delta)}\cP_{t-r}\varphi\right) (Y_{r}^\delta) \right)dr.
\end{align*}
To study the first addend, i.e. the term $I_1$, we 
fix $u\geq 0$ then we apply Lemma \ref{lem:ItoforY} to the time-independent function $f(y)=\cL_{(y)}(\cP_{u-r}\varphi)(y)$, obtaining
\begin{equation*}
I_1=\E \int_0^t \left(  \cL_{(Y_{r}^\delta)}(\cP_{u-r}\varphi)(Y_r^\delta)- \cL_{(Y_{t_{n(r)}}^\delta)}\cP_{u-r}\varphi (Y_{t_{n(r)}}^\delta) \right)dr  = \E \int_0^t\int_{t_{n(r)}}^r  \cL_{(Y_{t_{n(r)}}^\delta)}(\cL\cP_{u-r}\varphi)(Y_s^\delta)ds dr.
\end{equation*}
By setting $u=t$ we have
\begin{equation*}
    I_1 = \E \int_0^t\int_{t_{n(r)}}^r  \cL_{(Y_{t_{n(r)}}^\delta)}(\cL\cP_{t-r}\varphi)(Y_s^\delta)ds dr.
\end{equation*}
We can control the right hand side of the above using Hypothesis \ref{hyp:euler} \ref{ass:lipschitz}; indeed
\begin{align*}
\lvert I_1\rvert &= \left\lvert\E \int_0^t \left(  \cL_{(Y_{r}^\delta)}(\cP_{t-r}\varphi)(Y_r^\delta)- \cL_{(Y_{t_{n(r)}}^\delta)}\cP_{t-r}\varphi (Y_{t_{n(r)}}^\delta) \right)dr\right\rvert \\
& \leq  \E \int_0^t\int_{t_{n(r)}}^r  \lvert \cL_{(Y_{t_{n(r)}}^\delta)}(\cL\cP_{t-r}\varphi)(Y_s^\delta) \rvert ds dr\\
&\leq \E \int_0^t\int_{t_{n(r)}}^r   \sum_{i=1}^N\lvert U_0^i(Y_{t_{n(r)}}^\delta)(\partial_i\cL\cP_{t-r}\varphi)(Y_s^\delta)\rvert \\
&+ \sum_{k=1}^d\sum_{i,j=1}^N\lvert V_k^i(Y_{t_{n(r)}}^\delta)V_k^j(Y_{t_{n(r)}}^\delta)(\partial_{i}\partial_j\cL\cP_{t-r}\varphi)(Y_s^\delta) \rvert ds dr\\
&\leq K \E \int_0^t\int_{t_{n(r)}}^r   \sum_{i=1}^N (1+\lvert Y_{t_{n(r)}}^\delta\rvert^p) \lvert  (\partial_i\cL\cP_{t-r}\varphi)(Y_s^\delta)\rvert + \sum_{k=1}^d\sum_{i,j=1}^N(1+\lvert Y_{t_{n(r)}}^\delta\rvert^{2p})\lvert (\partial_{i}\partial_j\cL\cP_{t-r}\varphi)(Y_s^\delta) \rvert ds dr \, ,
\end{align*}
 where $K$ is a positive constant which depends on $K_1$ and $K_2$.  Let us start by analysing the first addend on the right hand side of the above: 

\begin{align*}
&\E \int_0^t\int_{t_{n(r)}}^r   \sum_{i=1}^N (1+\lvert Y_{t_{n(r)}}^\delta\rvert^p) \lvert  (\partial_i\cL\cP_{t-r}\varphi)(Y_s^\delta)\rvert ds dr \\
&\leq \E \int_0^t\int_{t_{n(r)}}^r   \sum_{i=1}^N (1+\lvert Y_{t_{n(r)}}^\delta\rvert^p)\left\lvert\partial_i\left(\sum_{\ell=1}^N U_0^\ell(\cdot)\partial_\ell\cP_{t-r}\varphi+\sum_{k=1}^d\sum_{\ell,m=1}^N V_k^\ell(\cdot)V_k^m(\cdot)\partial_\ell\partial_m\cP_{t-r}\varphi\right)(Y_s^\delta)\right\rvert ds  dr\\
&\leq \E \int_0^t\!\!\int_{t_{n(r)}}^r (1+\lvert Y_{t_{n(r)}}^\delta\rvert^p)  \sum_{i=1}^N\sum_{\ell=1}^N \Bigg( \lvert\partial_iU_0^\ell(Y_s^\delta)\partial_\ell\cP_{t-r}\varphi(Y_s^\delta)\rvert +\lvert U_0^\ell(Y_s^\delta)\partial_{i,\ell}\cP_{t-r}\varphi(Y_s^\delta)\rvert\\
&+ \!2\sum_{k=1}^d\!\sum_{m=1}^N \lvert\partial_iV_k^\ell(Y_s^\delta)V_k^m(Y_s^\delta)\partial_{\ell,m}\cP_{t-r}\varphi(Y_s^\delta)\rvert \!+\!\sum_{k=1}^d\!\sum_{m=1}^N \lvert V_k^\ell(Y_s^\delta)V_k^m(Y_s^\delta)\partial_{i,\ell,m}\cP_{t-r}\varphi(Y_s^\delta)\rvert \!\Bigg)dsdr.
\end{align*}
Now we use Hypothesis \ref{hyp:euler} \ref{ass:gradest} to estimate each of these terms.
\begin{align*}
&\E \int_0^t\int_{t_{n(r)}}^r  \sum_{i=1}^N (1+\lvert Y_{t_{n(r)}}^\delta\rvert^p) \lvert  (\partial_i\cL\cP_{t-r}\varphi)(Y_s^\delta)\rvert ds dr \\
&\leq K \E \int_0^t\int_{t_{n(r)}}^r  (1+\lvert Y_{t_{n(r)}}^\delta\rvert^p)\sum_{i,m,\ell=1}^N\bigg(  (1+\lvert Y_s^\delta\rvert ^q)\lvert\partial_\ell\cP_{t-r}\varphi(Y_s^\delta)\rvert +(1+\lvert Y_s^\delta\rvert^{p})(1+\lvert Y_s^\delta\rvert^q) \lvert \partial_{i,\ell}\cP_{t-r}\varphi(Y_s^\delta)\rvert\\
&+(1+\lvert Y_s^\delta\rvert^{2p})\lvert\partial_{i,\ell,m}\cP_{t-r}\varphi(Y_s^\delta)\rvert\bigg) dsdr\\
&\leq K \lVert \varphi \rVert_{C_b^4}\!\! \int_0^t\!\!\!\int_{t_{n(r)}}^r\!\!\! \mathbb{E}\left[(1+ \lvert Y_{t_{n(r)}}^\delta\rvert^p)(1+\lvert Y_s^\delta\rvert ^q +(1+\lvert Y_s^\delta\rvert^{p})(1+\lvert Y_s^\delta\rvert^q)+\lvert Y_s^\delta\rvert^{2p})u(Y_s^\delta)\right]e^{-\lambda_0 (t-r)} dsdr\\
&\leq K \lVert \varphi \rVert_{C^4_b} \delta.
\end{align*}
Here we have used estimate \eqref{eq:uboundinexp} and the fact that $r-t_{n(r)}\leq \delta$ to obtain the final inequality. In the above $K$ is a generic positive constant, the value of which changes line by line and only depends on $K_1,\ldots, K_6,\lambda, d, N$ but does not depend on $t$. Similarly by using estimate \eqref{eq:uboundinexp2} we obtain
\begin{align*}
\E \int_0^t\int_{t_{n(r)}}^r   \sum_{i,j=1}^N(1+\lvert Y_{t_{n(r)}}^\delta\rvert^{2p})\lvert (\partial_{i,j}\cL\cP_{t-r}\varphi)(Y_s^\delta)\rvert ds dr \leq K \lVert \varphi \rVert_{C^4_b} \delta\, .
\end{align*}
Putting everything together, one obtains
\begin{equation}
    \lvert I_1 \rvert \leq K \delta \lVert \varphi \rVert_{C^4_b} \delta. \label{eq:boundforI1}
\end{equation}
Now consider the term $I_2$; similarly to what we have done before, we use Lemma \ref{lem:ItoforY} applied to the function $f(y)=\cL_{(Y_{t_{n(r)}}^\delta)}\cP_{u-r}\varphi(y)$  and calculate the resulting expression when $u=t$. We then have
\begin{align*}
\lvert I_2\rvert&=\left\lvert\E \int_0^t \left(  \cL_{(Y_{t_{n(r)}}^\delta)}(\cP_{t-r}\varphi)(Y_{t_{n(r)}}^\delta)- \cL_{(Y_{t_{n(r)}}^\delta)}\cP_{t-r}\varphi (Y_{r}^\delta) \right)dr \right\rvert\\
&= \left\lvert\E \int_0^t\int_{t_{n(r)}}^r \cL_{(Y_{t_{n(r)}}^\delta)}\cL_{(Y_{t_{n(r)}}^\delta)}(\cP_{t-r}\varphi)(Y_{s}^\delta) ds dr \right\rvert\\
&=\bigg\lvert\E \!\! \int_0^t\!\!\!\int_{t_{n(r)}}^r \!\!\sum_{i,\ell=1}^NU_0^i(Y_{t_{n(r)}}^\delta)U_0^\ell(Y_{t_{n(r)}}^\delta)\partial_{i,\ell}(\cP_{t-r}\varphi)(Y_{s}^\delta)+ 2\!\!\sum_{i,j,\ell=1}^N \!\! U_0^\ell(Y_{t_{n(r)}}^\delta)V_k^i(Y_{t_{n(r)}}^\delta)V_k^j(Y_{t_{n(r)}}^\delta)\partial_{i,j,\ell}\cP_{t-r}\varphi(Y_s^\delta) \\
&+ \sum_{i,j,\ell,m=1}^N\sum_{k,q=1}^d V_k^i(Y_{t_{n(r)}}^\delta)V_k^j(Y_{t_{n(r)}}^\delta)V_q^\ell(Y_{t_{n(r)}}^\delta)V_q^m(Y_{t_{n(r)}}^\delta)\partial_{i,j,\ell,m}\cP_{t-r}\varphi(Y_s^\delta)ds dr \bigg\rvert\\
&\leq K\E \int_0^t\int_{t_{n(r)}}^r \sum_{i,\ell=1}^N(1+\lvert Y_{t_{n(r)}}^\delta\rvert^{2p})\lvert\partial_{i,\ell}(\cP_{t-r}\varphi)(Y_{s}^\delta)\rvert + \sum_{i,j,\ell=1}^N(1+\lvert Y_{t_{n(r)}}^\delta\rvert^{3p})\lvert\partial_{i,j,\ell}\cP_{t-r}\varphi(Y_s^\delta)\rvert\\
&+ \sum_{i,j,\ell,m=1}^N\sum_{k,q=1}^d (1+\lvert Y_{t_{n(r)}}^\delta\rvert^{4p})\lvert \partial_{i,j,\ell,m}\cP_{t-r}\varphi(Y_s^\delta)\rvert ds dr \,.
\end{align*}
Now we use estimate \eqref{eq:expdecayuptoorder4nonuniform} to obtain
\begin{equation}
\lvert I_2\rvert 
\leq K \lVert \varphi\rVert_{C^4_b} \int_0^t \int_{t_{n(r)}}^r\mathbb{E}\left[(1+ \lvert Y_{t_{n(r)}}^\delta\rvert^{4p})u(Y_s^\delta)\right]ds\, e^{-\lambda_0 (t-r)} dr\leq K \delta \lVert \varphi\rVert_{C^4_b} \,. \label{eq:boundforI2}
\end{equation}
To get the final inequality we have used \eqref{eq:uboundinexp3}. 
The proof is concluded by combining \eqref{eq:boundforI1}, \eqref{eq:boundforI2} and \eqref{eq:decompoferror}.
\end{proof}

\begin{corollary}\label{cor:ergodicaverages}
Suppose the coefficients of the SDE \eqref{eq:SDEito}  satisfy Hypothesis \ref{hyp:euler}. If the solution $X_t$ of the SDE \eqref{eq:SDEito} is ergodic with  invariant measure $\mu$ i.e. 
\begin{equation}
\left\vert \frac{1}{t} \int_0^t  \mathbb E \left[\varphi(X_s)\right] ds - \int_{\R^N} \varphi(x)d\mu(x) \right \vert  \longrightarrow 0 \quad \mbox{as } t \rightarrow \infty,
\end{equation}
 then 
\begin{equation}
\left\vert \frac{1}{t} \int_0^t  \mathbb E \left[\varphi(Y_s^{\delta})\right] ds - \int_{\R^N} \varphi(x)d\mu(x) \right \vert  \longrightarrow 0 \quad \mbox{as } t \rightarrow \infty, \delta \rightarrow 0,
\end{equation}
for every function $\varphi \in C_b^{\infty}$, where $Y_t^{\delta}$ has been defined in \eqref{eq:continterp}. 
\end{corollary}
\begin{proof}[Proof of Corollary \ref{cor:ergodicaverages}.] Note that if $X_t$ admits an invariant measure, then such an invariant measure is unique by ellipticity so the initial datum $x$ of the SDE doesn't play a role in what follows. Using Theorem \ref{thm:globalerror}, we have 
\begin{align*}
& \left\vert \frac{1}{t} \int_0^t  \mathbb E \left[\varphi(Y_s^{\delta})\right] ds - \int_{\R^N} \varphi(x)d\mu(x) \right \vert \\ 
& \leq  \frac{1}{t} \int_0^t  \left\vert  \mathbb{E} \varphi(Y_s^{\delta}) - \mathbb{E} \varphi(X_s)  \right\vert ds + \left\vert \frac{1}{t} \int_0^t \mathbb E \varphi(X_s)  ds  - \int_{\R^N} \varphi(x)d\mu(x)   \right \vert. 
\end{align*}
The first addend on the RHS tends to zero thanks to \eqref{eq:uniformweakconv}, the second one by assumption. 
    \end{proof}

\begin{example}\label{ex:additivenoiseeuler}\textup{Consider the one-dimensional SDE
    \begin{equation}\label{eq:SDEadditivenoise1dim}
      dX_t=b(X_t)dt+\sqrt{2} dB_t.    
    \end{equation}
     With the notation set so far, for this example we have   $V_0=U_0=b(x)\partial_x$ and $ V_1=\partial_x$. 
    Here $b:\R\to\R$ is a smooth function with bounded derivatives of all orders (but $b(x)$ itself is not assumed to be bounded). Suppose also that 
    \be\label{eq:bprimed}
    b'\leq -\lambda_0, \mbox{for some } \lambda_0>0. 
    \ee
    The process obviously satisfies Hypothesis \ref{hyp:euler} \ref{ass:ellipiticity}. 
    By Lemma \ref{lem:higherorderOACforadditivenoise} Hypothesis \ref{hyp:euler} \ref{ass:gradest} holds with $u(x)=constant$. Now, as in Note \ref{note: applicability}, if  \eqref{eq:bdd4moment} holds then Hypothesis \ref{hyp:euler} \ref{hypmom} is satisfied as well.
    To verify that \eqref{eq:bdd4moment} is satisfied, we  can integrate \eqref{eq:bprimed} and find that
    \be\label{eq:ybyineq}
    yb(y)\leq yb(0)-\lambda_0 y^2.
    \ee
     It is shown in \cite[Corollary 7.5]{Mattingly} that if $b$ is smooth, globally Lipschitz and satisfies \eqref{eq:ybyineq} then for sufficiently small $\delta>0$ there is a unique invariant measure $\pi^{\delta}$ for the numerical approximation $Y_{t_n}^\delta$. Moreover for each $\ell\geq 1$ there exist constants $C=C(\ell,\delta), \overline{\lambda}=\overline{\lambda}(\ell,\delta)>0$ such that for all functions $g$ such that $g(z) \leq C(1+\lvert z\rvert^{2\ell})$ we have  
    \begin{equation*}
    \left\vert\mathbb{E}[g(Y_{t_n}^\delta)]-\int_\R g(z) \pi^{\delta}(dz) \right\vert \leq C (1+\lvert x \rvert^{2\ell})e^{-\overline{\lambda} n \delta}, 
    \end{equation*}
    where $x$ is the initial datum of the SDE (and of the Euler approximation). 
    In particular, by taking $g(x)=x^4$ (i.e. $\ell=2$) we see that \eqref{eq:bdd4moment} is satisfied and we may apply Theorem \ref{thm:globalerror} to find that the weak error of the Euler scheme  converges to zero  uniformly in time.}\hf
\end{example}

\begin{corollary}\label{cor:additivenoisesummary}
Consider the SDE \eqref{eq:SDEadditivenoise1dim} and assume that $b$ is a smooth function with bounded derivatives of all orders. Moreover assume either one of the following:
\begin{enumerate}[(i)]
    \item There exists some constant $\lambda_0>0$ with $b'(y)\leq -\lambda_0$ for all $y\in \R$;\label{item:OACassumption}
    \item \label{item:weakestassumption}The function $b$ is bounded,  $-2b'(y)-\Xi(y)\geq 2 \lambda_0$ for all $y\in \R$,   $\mathrm{sign}(x)b(x)<0$ for $x$ sufficiently large and $\mathbb{E}[u(Y_{t_n}^\delta)]$ is bounded independently of $n$,  where \begin{align*}
        u(y)&=\cosh(\alpha y),\\
        \Xi(y) &= \alpha^2+\alpha b(y)\tanh(\alpha y),
    \end{align*}
    for some $\alpha>0$.
\end{enumerate}
Then Hypothesis \ref{hyp:euler} holds and by Theorem \ref{thm:globalerror} the weak error of the Euler approximation of \eqref{eq:SDEadditivenoise1dim} converges to zero uniformly in time.
\end{corollary}

\begin{proof}[Proof of Corollary \ref{cor:additivenoisesummary}]
 If we assume \ref{item:OACassumption} holds then the result follows from Example \ref{ex:additivenoiseeuler}. On the other hand, if we assume \ref{item:weakestassumption} holds then by Theorem \ref{thm:globalerror} it is sufficient to check Hypothesis \ref{hyp:euler} \ref{ass:gradest} and \ref{hypmom} hold. Hypothesis \ref{hyp:euler} \ref{ass:gradest} will follow from Theorem \ref{prop:OAC}, Theorem \ref{thm:Lyapunovbound} and Lemma \ref{lem:hypcheckforadditivenoiseex}. In particular, Theorem \ref{prop:OAC} can be applied after observing that the LOAC \eqref{LOAC1d} with $V_0=b(x) \pa_x$ and $V_1=V=\pa_x$ holds for this example once we take $\lambda(x)=-b'(x)$. Theorem \ref{thm:Lyapunovbound} can be  instead applied thanks to Lemma \ref{lem:hypcheckforadditivenoiseex}.  By  Note \ref{note: applicability}, Hypothesis \ref{hyp:euler} \ref{hypmom} reduces to \eqref{eq:uboundassumption}, which holds by assumption \ref{item:weakestassumption}.
\end{proof}

\section{A pathwise  approach to derivative estimates for Markov semigroups }\label{sec:pathwiseOAC}

In this section and the next we study derivative estimates for Markov semigroups, i.e. we study sufficient conditions in order for bounds of the type \eqref{expdecayintro}   to hold. To be more precise, in this section we find conditions in order for 
\eqref{gradest}  to hold, in Section \ref{sec:DonskerVaradhan} we will give criteria to obtain \eqref{expdecayintro} from \eqref{gradest}.  
We will consider SDEs of the form \eqref{eq:SDE} and, in order to explain ideas in a simple setting, we first consider the one-dimensional case  $N=1$ (Theorem \ref{prop:OAC}) and then generalise results to the case $N>1$ (Theorem \ref{prop:LOAChd}). If $N=1$ then,  Lemma \ref{lem:simplificationin1d} shows that  without loss of generality we may assume that $d=1$ as well and consider  one-dimensional SDEs of the form
\begin{equation}\label{eq:SDEmain1d}
	    dX_t^{(x)}=V_0(X_t^{(x)})+\sqrt{2}V_1(X_t^{(x)})\circ dB_t, \quad X_0^{(x)}=x \in \R.
	\end{equation}
	
\begin{remark}\label{Note:ellipticity}
\textup{Let us make some remarks on the relation between \eqref{expdecayintro} and \eqref{eq:expdecayuptoorder4nonuniform} and on the setting of this section and the next.  }
\begin{itemize}
\item \textup{In  Hypothesis \ref{hyp:euler} we require derivatives in the coordinate directions to decay exponentially fast, see \eqref{eq:expdecayuptoorder4nonuniform}. This is due to the fact that in Section \ref{sec:Eulerestimate} we were working in the setting in which the SDE at hand is elliptic. From this section on all the results we obtain are completely general in this respect and do not require any ellipticity to hold. We therefore study derivatives along more general vector fields. If in \eqref{expdecayintro} one takes $V(x)= \pa_x$ then \eqref{expdecayintro} becomes just \eqref{eq:expdecayuptoorder4nonuniform} (almost, see next bullet point).}
\item \textup{We shall concentrate on estimates for first order derivatives however similar arguments could be applied to higher order derivatives as shall be demonstrated in Lemma \ref{lem:higherorderOACforadditivenoise} for a class of examples.}
\item \textup{As we have already said, in this section we first consider the case $N=1$ and then move on to the general case $N>1$. When $N=1$, (under our assumptions) it suffices to study derivative estimates in the direction $V_1$. Let us explain why this is the case.  Suppose first that \eqref{eq:SDEmain1d} is uniformly elliptic.  We recall (see  Lemma \ref{lem:lamperti}), that  any one-dimensional uniformly elliptic SDE of the form \eqref{eq:SDEmain1d} can be transformed into a (one-dimensional) SDE with additive noise (i.e into an   SDE of the form  \eqref{eq:transformedSDE}). After such a transformation the differential operator $V_1$ is therefore just the derivative in the coordinate direction, $V_1=\pa_x$. Hence, in the  elliptic case, one can always recover derivative estimates in the coordinate direction $\pa_x$ from derivative estimates in the direction $V_1$. If the one-dimensional SDE \eqref{eq:SDEmain1d} is not uniformly elliptic it is not in general true that it suffices to study derivatives in the direction $V_1$. However, if $N=1$ (hence one can take $d=1$ as well) and we impose the LOAC \eqref{eq:LOACforV1}, it is indeed the case that it suffices to study the derivatives of the semigroup $\cP_t$ generated by \eqref{eq:SDEmain1d} just in the direction $V_1$; we prove this fact in Lemma  \ref{lem:onlyneedlevel1}.
}
 \end{itemize}
\hf
\end{remark}

While we do not  assume any particular growth condition on  the vector fields $V_0,V_1$, the case we really have in mind in developing this section and the next is the one in which the coefficients of the SDE are bounded. To explain why, in a way, this case is harder than the one in which one has some growth of the coefficients, let us start by recalling that in \cite{CrisanOttobre} the authors proved that, under the OAC \eqref{OAC1d},  the estimate \eqref{expdecayintro} follows with $u(x)=constant$ (precise statement in Appendix \ref{app:UFG}); however, as we show in Lemma \ref{lem:OACforaddnoise},  roughly speaking, the OAC implies unboundedness of the coefficients of the SDE.  
On the other hand, one does expect that exponential decay of derivatives of the semigroup may hold even if the coefficients of the SDE are bounded.  To illustrate why this is the case on a (relatively) simple example, start by considering the one-dimensional ODE 
\begin{equation}\label{eq:arctanODE}
    \frac{d}{dt}\xi_t^{(x)} =-\arctan(\xi_t^{(x)}), \quad \xi_0=x \,.
\end{equation}
This ODE has a single equilibrium at  $\xi=0$ and such an equilibrium is stable. Moreover,  for any $x\in\R$, we have \footnote{Differentiating \eqref{eq:arctanODE} with respect to $x$ gives
$$
\frac{d}{dt}\partial_x\xi_t^{(x)} = -\frac{1}{1+(\xi_t^{(x)})^2}\partial_x\xi_t^{(x)}.
$$
We can solve this to find
$$
\partial_x\xi_t^{(x)} = \exp\left(-\int_0^t \frac{1}{1+(\xi_s^{(x)})^2} ds\right).
$$
Finally, since $\xi_s^{(x)}$ converges monotonically towards zero we have $(\xi_s^{(x)})^2\leq x^2$ and hence \eqref{eq:arctanODEjacestimate} follows.
}  
\begin{equation}\label{eq:arctanODEjacestimate}
    \partial_x(\xi_t^{(x)}) \leq \exp\left(-\frac{t}{1+x^2}\right). 
\end{equation}
Motivated by this analogy we shall consider the SDE
\begin{equation}\label{eq:arctanintro}
    dX_t^{(x)} = -\arctan(X_t^{(x)})dt+\sqrt{2}dB_t.
\end{equation}
In Example \ref{ex:arctanprelim} and Example \ref{ex:arctan} we will show that \eqref{expdecayintro} does hold for the above SDE (and moreover that the Euler method does weakly approximate \eqref{eq:arctanintro} uniformly in time). Although this example {\em does not} satisfy the OAC \eqref{OAC1d},  one can easily verify  that for each $R>0$ and $f$ sufficiently smooth we have
\begin{equation*}\label{eq:OACarctan}
([V_{1},V_0]f)(x)(V_{1}f)(x) \leq -\frac{1}{1+R^2}\lvert (V_{1}f)(x)\rvert^2, \quad \mbox{for every } x\in [-R,R].
\end{equation*}
That is, the OAC is locally satisfied for $x\in [-R,R]$; this motivates us to introduce local versions \eqref{LOAC1d} of the OAC. 
\begin{remark}\label{Note:expmomentsarctangent}\textup{
We note in passing that the solution of  \eqref{eq:arctanintro} has uniformly in time bounded exponential moments, i.e.
\begin{equation*}
    \sup_{t\geq 0} \mathbb{E}[e^{\lvert X_t^{x}\rvert} ]<\infty, \quad \forall x\in\R, 
\end{equation*}
see Lemma \ref{lem:expmomentsforarctan}. 
So, overall,  on any fixed interval we have a version of the Obtuse Angle Condition and the probability of the process leaving an interval is exponentially small (for each $R>0$ the probability $X_t\notin [-R,R]$ is bounded by $Ce^{-R}$ by Markov's inequality).
\hf}
\end{remark}

Because of the local nature of \eqref{LOAC1d}, in this section we shall develop a pathwise approach to obtaining exponential decay \eqref{expdecayintro} of the derivative in direction $V_1$ of the semigroup under the condition \eqref{LOAC1d}. 

We now move on to proving that if the LOAC \eqref{LOAC1d} is satisfied with $V=V_1$, namely if
\begin{equation}\label{eq:LOACforV1}
    [V_1,V_0](x)V_1(x) \leq -\lambda(x) \lvert V_1(x)\rvert^2,
\end{equation}
then,  for every $t\geq 0, x\in\R, f\in \CVone$, one has 
		\begin{equation}\label{eq:gradest}
	\lvert V_1\cP_tf(x)\rvert \leq \mathbb{E}\left[\exp\left(-2\int_0^t\lambda(X_r^{(x)})dr\right) \right]^{\frac{1}{2}}\lVert V_1f\rVert_\infty, 
	\end{equation}
	where $\cP_t$ is the semigroup generated by \eqref{eq:SDEmain1d} and
	 $\CVone$ denotes the set of all smooth functions $f$ such that $\| V_1f \|_\infty$ is finite. 
	
	 We  denote by $\J_t=\J_t^x=\frac{\partial}{\partial x}X_t^{(x)}$ the  derivative of $X_t^{(x)}$ with respect to $x$; this (one dimensional process) exists by \cite[Theorem 7.3]{Kunita} and can be viewed as the solution of
\begin{equation}\label{eq:SDEforJacobian}
d\J_t^{x}= V_0'(\Xx)\J_t^{x} dt + \sqrt{2}V_1'(\Xx)\J_t^{x} \circ dB_t, \quad \J_0^{x}=1\,.
\end{equation}	
	With this notation in place, we  write derivatives of the semigroup in terms of derivatives of the process $X_t^{(x)}$.  
	\begin{lemma}\label{lem:derivativeofsemigp}
	    Let $\cP_t$ be the semigroup generated by the SDE \eqref{eq:SDEmain1d} and assume that the LOAC \eqref{eq:LOACforV1} is satisfied by the vector fields in \eqref{eq:SDEmain1d} with a function  $\lambda(x)$ such that  $\lambda(x)\geq -\kappa$ for every $x\in\R$, for some $\kappa\in \R$ (note that $\kappa$ need not be negative). Then 
	    \begin{align}
	V_1\cP_tf(x) = 
	\mathbb{E}[f'(\Xx) \J_t V_{1}(x)]\label{eq:devofsemigpintermsofV}
	\end{align}
	for every $x\in \R$ and $f\in \CVone$. For clarity we emphasize that here $f'(\Xx)$ denotes the derivative of $f$ evaluated at $\Xx$.
	\end{lemma}
	
	\begin{proof}
Fix $f\in \CVone$ and fix some initial condition $x\in \R$; then, by the chain rule, we have
\begin{equation*}
    V_1(f(X_t^{(x)})) = V_1(x)f'(X_t^{(x)})J_t.
\end{equation*}
Now we can take expectations to obtain
\begin{equation}\label{eq:intstepinderofsemigp}
   \mathbb{E}\left[V_1(f(X_t^{(x)}))\right] = \mathbb{E}\left[V_1(x)f'(X_t^{(x)})J_t\right].
\end{equation}
At the end of the proof of Theorem \ref{prop:OAC} we justify swapping the expectation and the derivative on the left hand side of the above equality.  After doing so we have \eqref{eq:devofsemigpintermsofV}.
	\end{proof}

Let us introduce the two parameter random process $\{\Gamma_{s,t}\}_{0\leq s\leq t}$, defined as follows:
	\begin{equation*}
	\Gamma_{s,t} = \left\lvert f'(\Xx)\J_t\J_s^{-1}V_1(X_s^{(x)})\right\rvert^2.
	\end{equation*}
	The significance of the process $\Gamma_{s,t}$ will be more clear in view of \eqref{eq:functiondefofgamma}.

For the time being notice that by \eqref{eq:devofsemigpintermsofV} we have
\begin{align*}
\lvert V_{1}\cP_tf(x) \rvert^2&\leq \mathbb{E}\left[\left\lvert  f'(\Xx)\J_tV_{1}(x)\right\rvert^2\right]= \mathbb{E}[\Gamma_{0,t}],
\end{align*}
and moreover, (using that $f$ belongs to $\CVone$) we may estimate $\Gamma_{t,t}$ by
\begin{equation*}
\Gamma_{t,t} =  \lvert V_{1}f(X_t^{(x)})\rvert^2\leq \lVert V_1f\rVert_\infty^2.
\end{equation*}	 
Hence to prove \eqref{eq:gradest} it is sufficient to prove the following inequality
	\begin{equation}\label{eq:requiredestimate}
	\mathbb{E}[\Gamma_{0,t}]\leq  \mathbb{E}\left[\exp\left(-2\int_0^t \lambda(X_s^{(x)})ds \right)\Gamma_{t,t}\right].
	\end{equation}
	Before proving \eqref{eq:requiredestimate}, we shall introduce some more notation. For each $\omega\in \Omega, s\leq t$ we may define the random flow map $\Phi_{s,t}:\R\to\R$ by
	\begin{equation*}
	    \Phi_{s,t}(x) := X_t^{(s,x)}, \quad t\geq s\geq 0.
	\end{equation*}
	Here $X_t^{(s,x)}$ denotes the solution to \eqref{eq:SDE} given that $X_s^{(s,x)}=x$. It is shown in \cite{Kunita} that for almost all $\omega\in\Omega$, $\Phi_{s,t}$ is a well-defined diffeomorphism from $\R$ to $\R$ and we shall denote by $J_{s,t}$ the derivative $\Phi_{s,t}'(X_s^{(x)})$. By differentiating the identity $\Xx = \Phi_{s,t}(X_s^{(x)})$ with respect to $x$, we have $J_t=\Phi_{s,t}'(X_s^{(x)})J_s$ and hence
	\begin{equation*}
	\J_{s,t}=\Phi_{s,t}'(X_s^{(x)})=\J_t\J_s^{-1}.
	\end{equation*}
	Analogously, if $f_{s,t}(\cdot) := f(\Phi_{s,t}(\cdot))$, then $f_{s,t}(X_s^{(x)}) := f(\Phi_{s,t}(X_s^{(x)}))$, so that $f_{s,t}'(X_s^{(x)})=f'(\Xx)\J_t\J_s^{-1}$ and we may write 
		\begin{equation}\label{eq:functiondefofgamma}
	\Gamma_{s,t} = \left\lvert V_{1}f_{s,t}(X_s^{(x)})\right\rvert^2.
	\end{equation}	
	
\begin{theorem}\label{prop:OAC}
	Assume the SDE \eqref{eq:SDEmain1d} has a unique strong solution for every initial datum $x \in \R$  and suppose that the Local Obtuse Angle Condition \eqref{eq:LOACforV1} is satisfied by the vector fields appearing in \eqref{eq:SDEmain1d} with $\lambda=\lambda(x)$ a continuous function such that  $\lambda(x)\geq -\kappa$ for every $x\in\R$ and some $\kappa\in\R$. Then \eqref{eq:gradest} holds.
\end{theorem}	
	
	\begin{remark} \textup{Some clarifications on the statement of the above theorem. } \textup{\begin{itemize}
	\item Because the initial profile $f(x)$ is assumed to be smooth and the coefficients of the equation are smooth as well, the derivative $V_1\cP_t f$ always makes sense. Corollary \ref{cor1} below deals with the case in which $f$ is not smooth but just continuous and bounded. 
	    \item As we have already explained, we will require further conditions to ensure that the right hand side of \eqref{eq:gradest} decays exponentially. We will give conditions under which the right hand side of \eqref{eq:gradest} decays exponentially in Section \ref{sec:DonskerVaradhan}.
	    \item Theorem \ref{prop:OAC} (and Theorem \ref{prop:LOAChd}) give sufficient conditions to estimate the derivatives of diffusion semigroups. As we have already mentioned, the techniques of this section, and in particular the proof of such theorems,  rely on a ``pathwise" approach. Such an approach is ``pathwise" in the sense that it is crucial here that one waits to take expectations until the latest possible moment. 
	\end{itemize} }\hf
	\end{remark}
	
\begin{proof}[Proof of Theorem \ref{prop:OAC}]	
	We will use \cite[Equation (2.63)]{Nualart} which, in our notation and setting can be written as
	\begin{equation}\label{eq:jacinv}
	    d(J_t^{-1}V(\Xx)) = J_t^{-1}[V_0,V](X_t^{(x)})dt+\sqrt{2}J_t^{-1}[V_1,V](X_t^{x})\circ dB_t,
	\end{equation}
	where $V:\R\to\R$ is any smooth vector field. By taking $V=V_1$ in \eqref{eq:jacinv}, we obtain
	\begin{align*}
	d\left(\J_t^{-1}V_{1}(\Xx)\right)^2=   2\J_t^{-1}[V_0,V_{1}](\Xx)V_{1}(\Xx)(\J_t^{-1})dt .
	\end{align*}
	Integrating from $0$ to $s$ and multiplying both sides by $f'(\Xx)^2 \J_t^2$ one gets
		\begin{align*}
	\left\lvert f'(\Xx)\J_t\J_s^{-1} V_{1}(\Xx) \right\rvert^2 &=\left\lvert f'(\Xx)\J_t V_{1}(x)\right\rvert^2\\
	&+ 2\int_0^s f'(\Xx) \J_t\J_r^{-1}[V_0,V_{1}](X_r^{(x)})V_{1}(X_r^{(x)})\J_r^{-1}\J_t f'(\Xx)dr.
	\end{align*}
	Now we may apply \eqref{eq:LOACforV1} and obtain
	\begin{align*}
	\left\lvert f'(\Xx)\J_t\J_s^{-1} V_{1}(\Xx) \right\rvert^2 \geq &\left\lvert f'(\Xx)\J_t V_{1}(x)\right\rvert^2\\
	&+ 2\int_0^s \lambda(X_r^{(x)})\left\lvert f'(\Xx) \J_t\J_r^{-1}V_{1}(X_r^{(x)})\right\rvert^2 dr.
	\end{align*}
We can rewrite this in terms of $\Gamma_{s,t}$ as
			\begin{align*}\label{eq:Gammadecomp}
	\Gamma_{s,t} &\geq \Gamma_{0,t}+ 2\int_0^s\lambda(X_r^{(x)}) \Gamma_{r,t} dr.
	\end{align*}
That is,
\begin{equation}\label{eq:Gammaineq}
    \exp\left(-2\int_0^s(\lambda(X_r^{(x)}))  dr \right)\Gamma_{s,t} \geq \Gamma_{0,t}. 
\end{equation}
Taking expectations and setting $s=t$ one obtains \eqref{eq:requiredestimate}. It remains to justify that we may swap the expectation and the derivative on the left hand side of \eqref{eq:intstepinderofsemigp}. This follows from the dominated convergence theorem provided we have that
	\begin{equation*}
	\sup_{x\in\R} \lvert V_{1}(f(X_t^{(x)}))\rvert 
	\end{equation*}
	is bounded by a constant which may depend on $t$. By setting $s=t$ in \eqref{eq:Gammaineq} we have
	\begin{equation*}
	\lvert V_1(f(X_t^{(x)}))\rvert ^2=\Gamma_{0,t} \leq \exp\left(-2\int_0^t(\lambda(X_r^{(x)}))  dr \right)\Gamma_{t,t}.  
	\end{equation*}
	We may bound the right hand side using $-\lambda(x)\leq \kappa$ and $\Gamma_{0,t}\leq \lVert V_1f\rVert^2$, this gives
	\begin{equation*}
	\lvert V_1(f(X_t^{(x)}))\rvert ^2=\Gamma_{0,t} \leq e^{2\kappa t} \lVert V_1 f\rVert^2.  
	\end{equation*}
This concludes the proof.
\end{proof}

We now state a simple consequence of  Theorem \ref{prop:OAC}, Corollary \ref{cor1}. We then give some simple examples  to which Theorem \ref{prop:OAC} can be applied. Before stating  Corollary \ref{cor1} we observe that \eqref{eq:gradest} holds for smooth functions only. Corollary \ref{cor1} allows one to state an analogous result for functions $f$ which are only continuous and bounded. We start by recalling a well-known short-time smoothing result: for any compact set $K$  there is a constant $c=c(K)$ such that 
\be\label{sta}
\lv V_1\cP_tf(x) \rv \leq \frac{c(K)}{t} \| f \|_{\infty}, \quad f \in C_b(\R), t \in (0,1]. 
\ee
Using the above and the semigroup property,  by the same argument as in \cite[Note 3.2]{CrisanOttobre}, we obtain what follows. Such smoothing estimates hold under very general assumptions on the coefficients of the SDE, for example they do hold under the UFG condition, see Definition \ref{defufg} and \cite{Nee} for an account on the matter (note that UFG processes include both elliptic and uniformly hypoelliptic processes). 

\begin{corollary}\label{cor1}
Consider the SDE \eqref{eq:SDE} and assume that the LOAC  \eqref{LOAC1d} and the smoothing property \eqref{sta} hold.  Then, for any $t_0>0$ and compact set $K$ we can find a constant $c_{t_0,K}$ such that
	   	\begin{equation*}
	\lvert V_1\cP_tf(x)\rvert \leq c_{t_0,K}\mathbb{E}\left[\exp\left(-2\int_0^{t-t_0}\lambda(X_r^{(x)})dr\right) \right]^{\frac{1}{2}}\lVert f\rVert_\infty, \quad \forall x\in K, f\in C_b(\R^N), t\geq t_0.
	\end{equation*}
\end{corollary}

The examples below illustrate the situation in which the OAC \eqref{OAC1d} does not hold but the LOAC \eqref{LOAC1d} does.
\begin{example}\label{ex:arctanprelim}\textup{
Consider the SDE
	\begin{equation}\label{eq:arctanSDE}
	dX_t=-\arctan(X_t)dt+\sqrt{2}dW_t.
	\end{equation}
	In this case $N=d=1$ and we have $V_0(x)=-\arctan(x)$, $V_1(x)=1$. Then the LOAC \eqref{LOAC1d} is satisfied with 
	\begin{equation*}
	    \lambda(x) = -\frac{[V_1,V_0](x)V_1(x)}{V_1(x)^2} = \frac{1}{1+x^2}.
	\end{equation*}
	In Figure \ref{fig:bplotforarctan} is a plot of $V_0$ and $\lambda$. Notice that because $\lambda(x)$ converges to $0$ as $x$ tends to $\pm \infty$ the Obtuse Angle Condition \eqref{OAC1d} does not hold. By Theorem \ref{prop:OAC} we have
			\begin{equation}\label{eq:arctannonexpest}
	\lvert \partial_x\cP_tf(x)\rvert \leq \mathbb{E}\left[\exp\left(-2\int_0^t\frac{1}{1+(X_r^{(x)})^2}dr\right) \right]^{\frac{1}{2}}\lVert \partial_xf\rVert_\infty.
	\end{equation}
	We will continue investigating this SDE in Example \ref{ex:arctan} where we will show that the weak error of the Euler approximation of \eqref{eq:arctanSDE} converges to zero uniformly in time. 
		\begin{figure}
		\centering
		\includegraphics[width=0.7\linewidth]{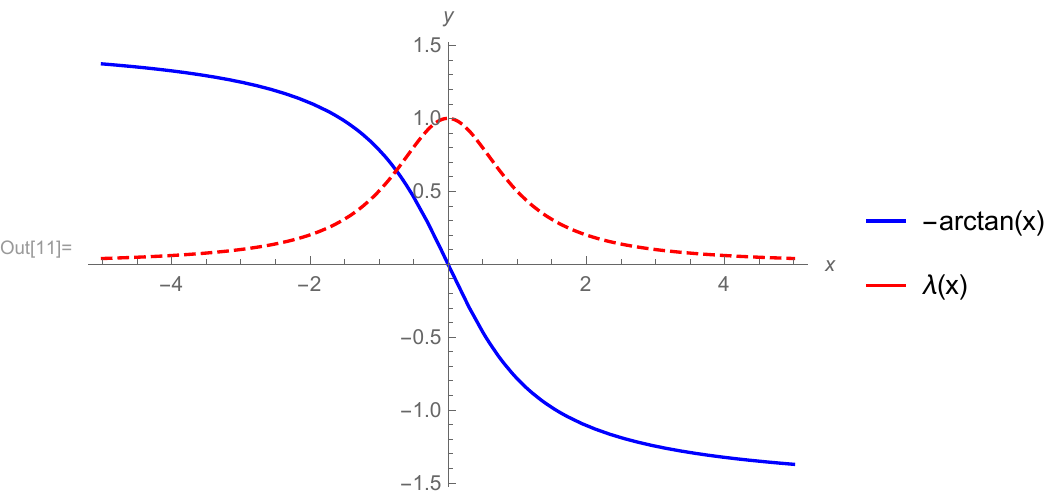}
		\caption{A plot of $b(x)$ and $\lambda(x)$ for the SDE \eqref{eq:arctanSDE}.}
		\label{fig:bplotforarctan}
	\end{figure}
	}\hf
\end{example}

\begin{example}\label{ex:sincos}
\textup{
Consider the one-dimensional SDE
\begin{equation}\label{eq:sincos}
dX_t=-\sin(X_t)dt+\sqrt{2}\cos(X_t)\circ dB_t.
\end{equation}
In this case we have $V_0(x)=-\sin(x)\partial_x$, $V_1(x)=\cos(x)\partial_x$, so that $[V_1, V_0]=-\pa_x$ and the LOAC \eqref{LOAC1d} is satisfied with
\begin{equation*}
\lambda(x)= \frac{1}{\cos(x)}.
\end{equation*}
 Here the OAC \eqref{OAC1d} is not satisfied (with $V=V_1$), indeed $\lambda$ is negative for $x\in (\pi/2,3\pi/2)$ and not defined for $x=k\pi+\pi/2$ for any $k\in\mathbb{Z}$. We also have that $\lambda(x)\geq 1$ for $x\in(-\pi/2,\pi/2)$. On the other hand, if $x\in (-\pi/2,\pi/2)$ then $\Xx\in(-\pi/2,\pi/2)$ (this can be seen directly from the SDE \eqref{eq:sincos} or see \cite[Excursus 4.5]{CCDO}). Therefore by Theorem \ref{prop:OAC} we have
\begin{equation*}
\lvert V_1\cP_tf(x)\rvert \leq e^{-t} \lVert V_1f\rVert_\infty, \quad \forall x\in\left(-\frac{\pi}{2}, \frac{\pi}{2}\right).
\end{equation*}
}\hf
\end{example}

\begin{remark} \label{compBakry}
\textup{ To simplify the discussion,  in this note we still consider  the simple setting of equation \eqref{eq:SDEmain1d}, i.e. we take \eqref{eq:SDE} with $d=N=1$. In \cite{CrisanOttobre} a Bakry-Emery type technique is used to prove that the OAC \eqref{OAC1d} (with $V=V_1$) implies estimates of the form \eqref{expdecayintro} (again with $V=V_1$). The argument used there (and in related literature) is a Gronwall-type argument and it fails if $\lambda=\lambda(x)$, i.e. if \eqref{LOAC1d} holds in place of \eqref{OAC1d}.  To explain why this is the case, we briefly recap the backbone of the argument used in \cite{CrisanOttobre} (and in related literature, see e.g. \cite{Dragoni, Bakry, MV_I}): let  
\begin{equation*}
    \Gamma(f):=\lvert V_1f(x)\rvert^2.
\end{equation*}
(Note that the above function $\Gamma(f)$ is the analogous of our $\Gamma_{s,t}$ in Theorem \ref{prop:OAC}).  The aim is to show the following inequality: 
\begin{equation}\label{eq:gronwallid}
    \partial_s \cP_{t-s}\Gamma(\cP_sf(x)) \leq -\lambda \cP_{t-s}\Gamma (\cP_sf (x)).
\end{equation}
Indeed, if the above holds, then the Gronwall lemma gives
\begin{equation*}
    \cP_{t-s}\Gamma (\cP_sf(x)) \leq
    e^{-\lambda} \cP_t \Gamma(f(x))
\end{equation*}
and the desired exponential decay of the derivative of the semigroup in the direction $V_1$ is obtained by just calculating the above in $s=t$. 
In order to obtain \eqref{eq:gronwallid} it is sufficient to prove (see \cite{CrisanOttobre}) the following inequality
\begin{equation}\label{eq:strategystandard}
    (\partial_t-\cL)\Gamma(\cP_tf(x)) \leq -\lambda \Gamma(\cP_tf(x)). 
\end{equation}
To prove the above the OAC was employed.  In the case when $\lambda=\lambda(x)$ we can follow the same argument and this time we obtain
\begin{equation*}
    (\partial_t-\cL)\Gamma(\cP_t f(x)) \leq -\lambda(x) \Gamma(\cP_t f(x)).
\end{equation*}
However instead of \eqref{eq:gronwallid} this implies
\begin{equation*}
    \partial_s \cP_{t-s}\Gamma(\cP_sf) \leq - \cP_{t-s}(\lambda (x)\Gamma(\cP_sf (x))).
\end{equation*}
Clearly, if $\lambda(x)$ is uniformly bounded below, then one can  use the previous argument again. If this is not the case then the Gronwall argument is no longer applicable.}
\hf
\end{remark}

\begin{lemma}\label{lem:higherorderOACforadditivenoise}
    Consider the SDE \eqref{eq:SDEadditivenoise1dim}; then \eqref{eq:expdecayuptoorder4nonuniform} holds for the semigroup generated by the process \eqref{eq:SDEadditivenoise1dim} provided the drift  $b(x)$ has bounded second, third, and fourth order derivatives, $b'(x)\leq 0$ and there is a positive constant $C>0$ such that
    \begin{equation}\label{eq:expdecayintermsofb}
    \mathbb{E}\left[\exp\left(\int_0^t b'(X_s^{(x)})ds\right)\right] \leq u(x) e^{-Ct} 
    \end{equation}
    for some positive function $u:\R\to \R$.
\end{lemma}
The proof of this lemma can be found in Appendix \ref{app:auxproofs}. 

We now extend the results of Theorem \ref{prop:OAC} to the higher dimensional setting, so from now on we consider the SDE \eqref{eq:SDE} with  $N\geq 1, d \geq 1$.  Fix some direction $V$ in which the semigroup $\cP_t$ generated by \eqref{eq:SDE}  is differentiable, i.e. such that $V\cP_tf$ makes sense for all $f\in \CV$ (here $\CV$ denotes the set of all $C^\infty$--functions $f$ such that $\lVert Vf \rVert_\infty$ is finite).
In this situation we prove that if the LOAC \eqref{LOAC1d} is satisfied and $[V,V_k]=0$ for all $k\in\{1,\ldots,d\}$ then for every $t\geq 0, x\in\R, f\in \CV$ we have 
\begin{equation}\label{eq:gradesthd}
\lvert V\cP_tf(x)\rvert \leq \mathbb{E}\left[\exp\left(-2\int_0^t\lambda(X_r^{(x)})dr\right) \right]^{\frac{1}{2}}\lVert V f\rVert_{\infty}.
\end{equation}

\begin{theorem}\label{prop:LOAChd}
	Let $\cP_t$ be the semigroup associated with the SDE \eqref{eq:SDE} and let  $V$ be a vector field along which $\cP_t$ is differentiable.  Assume that $[V,V_k]=0$ for all $k\in\{1,\ldots,d\}$ (where $V_1, \dots, V_d$ are the fields appearing in \eqref{eq:SDE}) and suppose that the Local Obtuse Angle Condition \eqref{LOAC1d} is satisfied by $V$ and $V_0$ with $\lambda$ a continuous function such that  $\lambda(x)\geq -\kappa$ for every $x\in\R$ and some $\kappa\in\R$. Then \eqref{eq:gradesthd} holds.
\end{theorem}

\section{Estimates for functionals of the occupation measure}\label{sec:DonskerVaradhan}

In Section \ref{sec:pathwiseOAC} we gave conditions under which the estimate \eqref{eq:gradesthd} holds. To obtain  exponential decay of derivatives it remains to find conditions under which there exists a constant $\lambda_0>0$ and a function $u:\R^N\to\R$ such that
\begin{equation}\label{eq:expdecay}
\mathbb{E}\left[\exp\left(-2\int_0^t \lambda(X_s^{(x)})ds \right)\right] \leq u(x) e^{-2\lambda_0 t}.
\end{equation}
This is the scope of this section. Clearly, a case under which the estimate \eqref{eq:expdecay}  follows immediately is the one in which the function $\lambda$ is bounded below by a positive constant i.e. $\lambda(x)\geq\lambda_0>0$. 

We can consider the weaker situation in which $\lambda\geq 0$ and there is some set $F$ such that $\lambda(x) \geq \lambda_F>0$ for some positive constant $\lambda_F$ and for every $x \in F$. Then we require that the process spends a positive proportion of time in the set $F$ (see Note \ref{Note:expmomentsarctangent}). More precisely, the following holds. 
\begin{prop}
Let $X_t^{(x)}$ be the solution of the SDE \eqref{eq:SDE}. Suppose that there exist some set $F\subseteq \R^N$ and a constant $r>0$ such that
		\begin{equation}\label{eq:recurrence}
		\frac{1}{t}\int_0^t \mathbbm{1}_F(X_s^{(x)}) ds \geq r \quad \mathbb{P}-a.s, \text{ for all } x\in\R^N.
		\end{equation}
		Let  $\lambda:\R^N\to\R$ be any  function\footnote{At this stage we do not assume that $\lambda(x)$ is the function appearing in the LOAC. } such that  $\lambda(x)\geq 0$ for every $x\in \R^N$ and there is a positive constant $\lambda_F$ such that $\lambda(x)\geq \lambda_F>0$ for all $x\in F$.
		Then, for all $t\geq 0$, we have
		\begin{equation}\label{eq:puttogether}
\mathbb{E}\left[\exp\left(-2\int_0^t \lambda(X_s^{(x)}) ds \right)\right] \leq \mathbb{E}\left[\exp\left(-2\int_0^t \lambda_F \mathbbm{1}_F(X_s^{(x)}) ds \right)\right]\leq \exp\left(-2r\lambda_F t\right).
		\end{equation}
Moreover, let $\cP_t$ be the semigroup associated with \eqref{eq:SDE} and $V$ a direction along which such a semigroup is differentiable. If, additionally, the vector field $V$ and the function $\lambda$ satisfy the assumptions of Theorem \ref{prop:LOAChd}, combining \eqref{eq:gradesthd} and \eqref{eq:puttogether}, one obtains	
\begin{equation*}
    \lvert V\cP_tf(x)\rvert \leq e^{-r\lambda_F t}
 \lVert Vf\rVert_\infty, \quad \text{ for all } f\in \CV, x\in\R^N,t\geq 0.
 \end{equation*}
\end{prop}
		We can view \eqref{eq:recurrence} as a form of recurrence. We can revisit this idea by using the large deviation principle for occupation measures introduced by Donsker and Varadhan. In a series of papers \cite{Donsker1}-\cite{Donsker4} Donsker and Varadhan introduced conditions to obtain a large deviation principle (LDP) for the {\em occupation measure} of $X_t^{(x)}$, i.e. for the random measure
\begin{equation}\label{OM}
l_t^x(\omega,A) = \frac{1}{t}\int_0^t \mathbbm{1}_{A}(X_s^{(x)}(\omega))ds.
\end{equation}
We briefly recall that the occupation measure $l_t^x$ satisfies a large deviation principle if there exists a rate function $I:\mathcal{M}\to \R$ such that
\begin{align}
    \limsup_{t\to\infty} \frac{1}{t}\log(\mathbb{P}(l_t^x\in C)) &\leq -\inf_{\mu\in C}I(\mu), \quad \text{for all closed sets } C\subseteq \mathcal{M}\label{eq:largedeviationsupper}\\
    \limsup_{t\to\infty} \frac{1}{t}\log(\mathbb{P}(l_t^x\in \mathcal{O})) &\leq -\inf_{\mu\in \mathcal{O}}I(\mu), \quad \text{for all open sets } \mathcal{O}\subseteq \mathcal{M}\label{eq:largedeviationslower}.
\end{align}
Note that $(\Omega,\mathcal{F},\mathbb{P})$ is the probability space on which the stochastic process $X_t$ is defined. Here $\mathcal{M}$ is endowed with the weak topology. We do not give details on this notion and refer the reader to \cite{Donsker1}-\cite{Donsker4}. For our purpose it is important to recall that if the occupation measure satisfies a LDP with rate function $I:\mathcal{M} \to \R$ (here $\mathcal{M}$ denotes the space of probability measures on $\R$) then for any weakly continuous functional\footnote{A functional $\Psi:\mathcal{M}\to \R$ is weakly continuous if given a sequence of measures $\mu_k$ which converge to a probability measure $\mu$ in the weak topology then $\Psi(\mu_k)$ converges to $\Psi(\mu)$.} $\Psi:\mathcal{M}\to\R$ and compact set $K\subseteq \R$, we have
\begin{equation}\label{eq:exptailforfunctional}
\lim_{t\to\infty} \frac{1}{t} \log \sup_{x\in K}\int_\Omega\exp\left(-t\Psi(l_t^x(\omega, \cdot)\right)\mathbb{P}(d\omega) = -\inf_{\mu\in \mathcal{M}}[\Psi(\mu)+I(\mu)].
\end{equation}
If $\lambda:\R^N\to\R$ is a continuous function we may take $\Psi:\mathcal{M}\to\R$ to be
\begin{equation*}
\Psi(\mu) = \int_{\R^N} \lambda(y) \mu(dy).
\end{equation*}
Then \eqref{eq:exptailforfunctional} becomes
\begin{equation}\label{eq:largedeviation}
\lim_{t\to\infty} \frac{1}{t} \log\sup_{x\in K} \mathbb{E}\left[\exp\left(-2\int_0^t\lambda(X_s^{(x)}) ds\right)\right] = -\inf_{\mu\in \mathcal{M}}\left[\int_{\R^N} 2\lambda(y) \mu(dy)+I(\mu)\right].
\end{equation}
\begin{prop}\label{prop:directfromDV} Let $X_t^{(x)}$ be the solution of the SDE \eqref{eq:SDE}. 
Suppose the occupation measure \eqref{OM} satisfies a LDP with rate function $I$ and assume there is a continuous function $\lambda:\R^N\to\R$ such that \eqref{eq:gradesthd} holds for some vector field $V$. If
\begin{equation}\label{eq:DVrhs}
\inf_{\mu\in M}\left[2\int_{\R^N} \lambda(y) \mu(dy)+I(\mu)\right] >0
\end{equation}
then   for each compact set $K\subset\R^N$ there exists a constant $C_K>0$ such that
\begin{equation}\label{eq:fullderivativeestimate}
    \sup_{x\in K}\lvert V\cP_tf(x)\rvert\leq C_K e^{-\lambda_0 t} \lVert Vf\rVert_\infty, \quad \forall f\in \CV,  
\end{equation}
for some $\lambda_0>0$ (independent of the compact set $K$). 
\end{prop}
 We recall that in \cite{Donsker4} a set of conditions is given in order for the occupation measure to satisfy a LDP. These are stated in  Hypothesis \ref{hyp:Donskerhyp} below.\footnote{Here such conditions are stated in our notation and setting.}

\begin{hypothesis}\label{hyp:Donskerhyp}
Let $\Xx$ be the solution of \eqref{eq:SDE} and $\cL$ be the corresponding generator.
\begin{enumerate}
\item There exists a function $\DV:\R^N\to \R$ and a sequence $u_n\in D(\cL)$ (here $D(\cL)$ denotes the domain of the operator $\cL:D(\cL)\subseteq C_b(\R^N) \to C_b(\R^N)$) such that the following properties hold:\label{hyp:Donskerupperbound}
\begin{enumerate}[(\ref{hyp:Donskerupperbound}a)]
\item The set $\{x\in\R^N:\DV(x)\geq \ell\}$ is compact for each $\ell\in\R$;\label{ass:compact}
\item For all $n\in\N, x\in\R^N$ we have $u_n(x)\geq 1$; \label{ass:lowerboundun}
\item For each compact set $K\subseteq \R^N$,\label{ass:upperboundforun}
\begin{equation*}
\sup_{x\in K}\sup_{n\in\N} u_n(x)<\infty;
\end{equation*}
\item For each $x\in\R^N$, \label{ass:limitdefofV}
\begin{equation}\label{eq:limitdefofV}
\lim_{n\to\infty}\frac{\cL u_n(x)}{u_n(x)} = \DV(x);
\end{equation}
\item For some $A<\infty$ \label{ass:boundsforLun}
\begin{equation}\label{eq:boundsforLun}
\sup_{n\in\N, x\in\R^N}\frac{\cL u_n(x)}{u_n(x)} \leq A;
\end{equation}
\end{enumerate}
\item Assume that the law of $X_t^{(x)}$ admits a density  $p(t,x,y)$ with respect to Lebesgue measure on $\R^N$ such that for all $x\in\R^N$:\label{hyp:Donskerlowerbound}
\begin{enumerate}[(\ref{hyp:Donskerlowerbound}a)]
\item $p(1,x,y)>0$ for almost all $y\in\R^N$ \label{ass:denpos}
\item The map $x\mapsto p(1,x,\cdot)$ is a continuous map from $\R^N$ to $L^1$. \label{ass:dencts}
\end{enumerate}
\end{enumerate}
\end{hypothesis}

\begin{remark}\label{rem:Donskerhyp}
\begin{itemize}\textup{
Let us comment on the above hypothesis.
    \item The first set of assumptions, Hypothesis \ref{hyp:Donskerhyp} \ref{ass:compact}--\ref{ass:boundsforLun}, are sufficient for an upper bound in the large deviation principle to hold, i.e. there is a rate function $I:\mathcal{M}\to \R$ such that \eqref{eq:largedeviationsupper} holds. One strategy to construct the sequence $u_n$ appearing in Hypothesis \ref{hyp:Donskerhyp} is as follows: first we find a pair of functions $u,\DV:\R^N\to\R$ such that
\begin{equation*}
    \cL u(x) = \DV(x)u(x)
\end{equation*}
and we require that $u\geq 1$, $\DV(x)$ is bounded above but tends to $-\infty$ as $\lvert x\rvert\to\infty$; we then construct the sequence $\{u_n\}_{n\in\N}$ by defining $u_n(x) = u(n\theta( x/n))$ where $\theta:\R^N\to\R^N$ is a smooth function such that for each component $i\in\{1,\ldots,d\}$ we have $\theta(-x)^i=-\theta(x)^i$ and 
\begin{equation*}
	\theta(y)^i = \begin{cases}
	y^i, & 0\leq y^i \leq 1;\\
	\text{smooth and increasing,} & 1\leq y^i\leq 2;\\
	2, & y^i\geq 2.
	\end{cases}
	\end{equation*}}
\textup{	The second set of assumptions, Hypothesis \ref{hyp:Donskerhyp} (\ref{hyp:Donskerlowerbound}), are sufficient for a lower bound in the large deviation principle, i.e. under Hypothesis \ref{hyp:Donskerhyp} \ref{ass:denpos}--\ref{ass:dencts} there is a rate function $I:\mathcal{M}\to \R$ such that \eqref{eq:largedeviationslower} holds. Note that in the case when \eqref{eq:SDE} satisfies a uniform ellipticity condition, i.e. there is some constant $\nu>0$ such that $V_1(x) \geq \nu>0 $ for all $x\in\R$, then Hypothesis \ref{hyp:Donskerhyp} \ref{ass:denpos}--\ref{ass:dencts} are satisfied (in contrast, under the weaker UFG condition -- see Appendix \ref{app:UFG} -- this latter set of assumptions is not satisfied).
\item Note that Hypothesis \ref{hyp:Donskerhyp} \ref{ass:compact} implies that $\DV$ is not bounded below, while Hypothesis \ref{hyp:Donskerhyp} \ref{ass:limitdefofV}  and Hypothesis \ref{hyp:Donskerhyp} \ref{ass:boundsforLun} imply that $\DV$ is bounded above by $A$.}
\end{itemize}
\hf
\end{remark}

By \cite[Theorem 7.2 and Theorem 8.1]{Donsker3} under Hypothesis \ref{hyp:Donskerhyp} the limit in \eqref{eq:largedeviation} holds with
\begin{equation}\label{eq:Idef}
I(\mu) = \sup_{u\in D(\cL), u>0} -\int_{\R^N}\frac{\cL u}{u} d\mu.
\end{equation} 
Note that the rate function $I$ is always non-negative (just take $u=const$). 
In order to prove that \eqref{eq:gradesthd} implies \eqref{expdecayintro} when Hypothesis \ref{hyp:Donskerhyp} is satisfied it remains to show that the right hand side of \eqref{eq:largedeviation} is positive. Note that by Fatou's lemma and \eqref{eq:boundsforLun} we have
 \begin{equation*}
I(\mu) \geq \liminf_{n\to\infty}\int_{\R^N} -\frac{\cL u_n}{u_n} d\mu  \geq \int_{\R^N} -\DV d\mu.
\end{equation*}
In particular
\begin{equation*}\label{eq:sufficientforfullDosnker}
\inf_{\mu\in M}\left[\int_{\R^N} \lambda(y) \mu(dy)+I(\mu)\right] \geq \inf_{\mu\in M} \int_{\R^N} (\lambda(y) -\DV(y)) \mu(dy).
\end{equation*}

We have therefore proven the following.
\begin{prop}\label{prop:simpleDVprop}
 Let $X_t^{(x)}$ be the solution of the SDE \eqref{eq:SDE} with $X_0^{(x)}=x$.   Assume that Hypothesis \ref{hyp:Donskerhyp} holds and there exist some continuous function  $\lambda:\R^N\to\R$ and a constant $\lambda_0>0$ such that $2\lambda(x)-\DV(x)\geq 2\lambda_0$ for all $x\in\R^N$. Then for each compact set $K\subseteq \R^N$ there is a constant $C_K$ such that
    \begin{equation*}
    \sup_{x\in K}\mathbb{E}\left[\exp\left(-2\int_0^t\lambda(X_r^{(x)})dr\right) \right] \leq C_Ke^{-2\lambda_0 t}, \quad \forall t\geq 0.
    \end{equation*}
\end{prop}

\begin{remark}\textup{
Note that since $\DV$ tends to $-\infty$ as $x\to \pm\infty$, for $\lvert x\rvert$ sufficiently large we have $\DV(x)<0$ in which case the condition $2\lambda-\DV\geq 2\lambda_0$ is weaker than the requirement that $\lambda\geq \lambda_0>0$ for $\lvert x\rvert$ sufficiently large. In Example \ref{ex:bump} we illustrate a case in which we are able to find a constant $\lambda_0>0$ such that  $2\lambda(x)-\DV(x)>2\lambda_0$ for all $x\in\R$ but $\lambda(x_0)<0$ for some $x_0\in \R$. }
\end{remark}

 Hypothesis \ref{hyp:Donskerhyp} is stronger than we require in order to control  $\lvert V\cP_tf(x)\rvert$. Indeed all we require is an upper bound for the left hand side of \eqref{eq:largedeviation} and we can achieve this under the following conditions.
\begin{hypothesis}\label{hyp:upperbound}
\begin{enumerate}
	\item With the same notation and setting as Hypothesis \ref{hyp:Donskerhyp}, there exist a function $\DV:\R^N\to \R$ and a sequence $u_n\in D(\cL)$ such that conditions \ref{ass:lowerboundun} - \ref{ass:boundsforLun} of Hypothesis \ref{hyp:Donskerhyp} hold.
	\item There exist a constant $\lambda_0>0$ such that
	\begin{equation}\label{eq:Vgap}
	2\lambda(x)-\DV(x) \geq 2\lambda_0>0.
	\end{equation}
\end{enumerate}
\end{hypothesis}

In particular we are no longer assuming that $\DV$ is unbounded from below,  which was required by Hypothesis \ref{hyp:Donskerhyp} \ref{ass:compact} (see Note \ref{rem:Donskerhyp}); instead, we require the existence of some constant $\lambda_0>0$ such that \eqref{eq:Vgap} holds.
Hypothesis \ref{hyp:upperbound} (1) is weaker than Hypothesis \ref{hyp:Donskerhyp} and the price we pay is that \eqref{eq:Vgap} is harder to satisfy than when $\DV$ was unbounded, however we will see in Example \ref{ex:arctan} that Hypothesis \ref{hyp:upperbound} is satisfied although Hypothesis \ref{hyp:Donskerhyp} is not.

\begin{theorem}\label{thm:Lyapunovbound}
	Assume that Hypothesis \ref{hyp:upperbound} holds for the SDE \eqref{eq:SDE}.
	Then \eqref{eq:expdecay} holds
with $u(x):=\liminf_{n\to\infty}u_n(x)$ (where $\{u_n\}$ is the sequence appearing in Hypothesis \ref{hyp:upperbound}).

Moreover, let $\cP_t$ be the semigroup associated with \eqref{eq:SDE} and $V$ a direction along which such a semigroup is differentiable. If, additionally, \eqref{eq:gradesthd} holds\footnote{If the vector field $V$ and the function $\lambda$ satisfy the assumptions of Theorem \ref{prop:LOAChd} then \eqref{eq:gradesthd} holds.} then combining \eqref{eq:gradesthd} and \eqref{eq:expdecay}, one obtains	
\begin{equation}\label{eq:fullderivativeestimatewithu}
    \lvert V\cP_tf(x)\rvert^2 \leq u(x) e^{-2\lambda_0 t}\lVert Vf\rVert_\infty^2, \quad \text{ for all } f\in \CV, x\in \R^N, t\geq 0 \,.
\end{equation}
\end{theorem}

\begin{proof}[Proof of Theorem \ref{thm:Lyapunovbound}]
Define
\begin{equation*}
\psi_n(x,t) = \mathbb{E}\left[u_n(X_t^{(x)}) \exp\left(-\int_0^t \frac{\cL u_n(X_s^{(x)})}{u_n(X_s^{(x)})} ds\right)\right].
\end{equation*}
By the Feynmann Kac formula, $\psi_n$ solves the initial value problem
\begin{equation}\label{eq:FKeqn}
\begin{cases}
&\frac{\partial \psi_n}{\partial t} = \cL \psi_n - \frac{\cL u_n}{u_n} \psi_n\\
&\psi_n(x,0)=u_n(x).
\end{cases}
\end{equation}
Note that $u_n$ is also a stationary solution to this PDE, indeed
\begin{align*}
    \cL u_n - \frac{\cL u_n}{u_n} u_n &=\cL u_n-\cL u_n=0.
\end{align*}
By \cite[Theorem 5.7.6]{KaratzasShreve} there is at most one solution to \eqref{eq:FKeqn} in the class $C^{1,2}(\R^N\times [0,T];\R)$ for each $T>0$ and hence we have $\psi_n(x,t)=u_n(x)$, that is
\begin{equation*}
u_n(x) = \mathbb{E}\left[u_n(X_t^{(x)}) \exp\left(-\int_0^t \frac{\cL u_n(X_s^{(x)})}{u_n(X_s^{(x)})} ds\right)\right].
\end{equation*} 
Using that $u_n\geq 1$ we have
\begin{equation*}
u_n(x) \geq \mathbb{E}\left[\exp\left(-\int_0^t \frac{\cL u_n(X_s^{(x)})}{u_n(X_s^{(x)})} ds\right)\right].
\end{equation*} 
By Fatou's lemma
\begin{align*}
u(x) = \liminf_{n\to\infty} u_n(x) &\geq  \mathbb{E}\left[\liminf_{n\to\infty}\exp\left(-\int_0^t \frac{\cL u_n(X_s^{(x)})}{u_n(X_s^{(x)})} ds\right)\right]
\end{align*}
Now using the continuity of the function $\exp$ we can exchange the $\liminf$ and $\exp$
\begin{align*}
u(x)\geq  \mathbb{E}\left[\exp\left(-\limsup_{n\to\infty}\int_0^t \frac{\cL u_n(X_s^{(x)})}{u_n(X_s^{(x)})} ds\right)\right]
\end{align*}
Again by reverse Fatou's lemma which is justified by \eqref{eq:boundsforLun}
\begin{align*}
u(x)&\geq  \mathbb{E}\left[\exp\left(-\int_0^t\limsup_{n\to\infty} \frac{\cL u_n(X_s^{(x)})}{u_n(X_s^{(x)})} ds\right)\right]\\
&=  \mathbb{E}\left[\exp\left(-\int_0^t \DV(X_s^{(x)}) ds\right)\right]
\end{align*}
here we have used \eqref{eq:limitdefofV} to justify the last line. Now using \eqref{eq:Vgap} we have
\begin{equation*}
u(x) \geq \mathbb{E}\left[\exp\left(-2\int_0^t \lambda(X_s^{(x)}) ds + 2\lambda_0t\right)\right].
\end{equation*} 
That is, 
\begin{equation*}
\mathbb{E}\left[\exp\left(-2\int_0^t \lambda(X_s^{(x)}) ds\right)\right] \leq u(x) e^{-2\lambda_0t}
\end{equation*}
as required.
\end{proof}

\section{Examples and Counterexamples}\label{sec6}

\begin{example}\label{ex:additivenoiseex}\textup{
Consider again the SDE \eqref{eq:SDEadditivenoise1dim}. 
If $b'(x)\leq -\lambda_0 <0$ for some constant $\lambda_0>0$ then one can deduce exponential decay of the derivatives of the semigroup from the results of \cite{CrisanOttobre}. Here we prove that the derivative estimates  \eqref{eq:fullderivativeestimate} hold also when $b'\leq 0$. More precisely, assuming $b(x)$ is unbounded (both above and below),  we show below the two following facts: i) if $b'(x)<0$ for every $x$ then \eqref{eq:fullderivativeestimate} holds for $V=V_1=\partial_x$; ii) if $b'(x)\leq 0$, then the same conclusion holds, provided Hypothesis \ref{hyp:Donskerhyp} is satisfied with some $\DV$ such that  $\DV(x)<0$ for all $x$ where $b'(x)=0$. 
  An example of a function $b(x)$ which falls in the case i) is $b(x)=\arctan(x)\log(2+x^2)$.}

\textup{
For equation \eqref{eq:SDEadditivenoise1dim} we have  $V_0(x)=b(x)\partial_x, V_1(x)=\partial_x$. The Local Obtuse Angle Condition \eqref{LOAC1d} is satisfied with $\lambda(x)=-b'(x)$, therefore by Theorem \ref{prop:OAC} \eqref{eq:gradest} holds. However since $b'$ is not necessarily uniformly bounded away from zero we do not immediately obtain \eqref{eq:fullderivativeestimate}; in order to obtain exponential decay we instead use the strategy of Section \ref{sec:DonskerVaradhan}. In Lemma \ref{lem:hypcheckforadditivenoiseex} we show that Hypothesis \ref{hyp:Donskerhyp} holds for \eqref{eq:SDEmain1d} when $b'(x)<0$ for all $x\in\R$.	By using Proposition \ref{prop:directfromDV}, in order to obtain \eqref{eq:fullderivativeestimate} it is then sufficient to show
	\begin{equation*}
	\lambda_0:=\inf_{\mu\in \mathcal{M}}\left[ I(\mu) -2\int_\R b' d\mu\right] >0,
	\end{equation*}
	where we recall that $I$ was given by \eqref{eq:Idef}. To prove the above suppose, for a contradiction, that $\lambda_0 =0$; then there exists some sequence of probability measures $\{\mu_k\}_{k\in\N}$ such that
	\begin{equation*}
	I(\mu_k) -2\int_\R b'(y) \mu_k(dy)\leq \frac{1}{k}
	\end{equation*}
	for every $k\in\N$. Now by Markov's inequality,
	\begin{equation*}
	\mu_k(\{x\in\R: A-\DV(x) >A+\ell\}) \leq \frac{A-\int_\R \DV d\mu_k}{A+\ell}
	\end{equation*}
	where $\DV$ and $A$ are as in Hypothesis \ref{hyp:Donskerhyp}, so that $\DV(x)\leq A$ for all $x\in\R$ and Markov's inequality is applicable. By the definition of $I$ and Fatou's lemma we have
	\begin{equation*}
	I(\mu_k) \geq \liminf_{n\to\infty}\int_\R -\frac{\cL u_n}{u_n} d\mu_k  \geq \int_\R -\DV d\mu_k.
	\end{equation*}
	This gives
	\begin{equation*}
	\mu_k\{x\in\R: \DV(x) \leq -\ell\} \leq \frac{A+I(\mu_k)}{A+\ell},
	\end{equation*}
	which implies that $\{\mu_k\}$ is tight since $\{x\in\R: \DV(x)\leq -\ell\}$ is compact for all $\ell$. By Prokhorov's theorem we may take a weakly convergent subsequence; let $\mu$ denote the limit of such a subsequence. Then
	\begin{equation}\label{eq:derivofbiszero}
	\int b'(y) \mu(dy) =0.
	\end{equation}
	However $b'<0$ so we have a contradiction. This proves that \eqref{eq:fullderivativeestimate} holds for the SDE \eqref{eq:SDEadditivenoise1dim}.}
	
	\textup{
	By following the same reasoning as in the above, we can also consider the case when $b'\leq 0$, provided Hypothesis \ref{hyp:Donskerhyp} holds for some $\DV$ such that  $\DV(x)<0$ for all $x$ where $b'(x)=0$. Indeed by \eqref{eq:derivofbiszero} we must have that $\mu(\{x:b'<0\})=0$. Therefore if $\DV(x)<0$ whenever $b'=0$ then we have
	\begin{equation*}
	0 = I(\mu)-2\mu(b') \geq \int_{\R} (2b'(y)-\DV(y)) d\mu(y) =  \int_{b'=0} -\DV d\mu>0
	\end{equation*} 
	which gives again a contradiction. }
	\hf
\end{example}

\begin{example}\label{ex:bump}\textup{
	Consider the SDE
	\begin{equation}\label{eq:bump}
	dX_t=(2\arctan(X_t-5)-X_t)dt+\sqrt{2}dB_t.
	\end{equation}
	For this example we will show that \eqref{eq:fullderivativeestimatewithu} holds. Indeed we have $V_0=(2\arctan(x-5)-x)\partial_x$, $V_1=\partial_x$, and then \eqref{eq:LOACforV1} is satisfied with 
	\begin{equation*}
	\lambda(x) = 1-\frac{2}{1+(x-5)^2}.
	\end{equation*}
	Now we may apply Theorem \ref{prop:OAC} and  see that \eqref{eq:fullderivativeestimatewithu} holds provided \eqref{eq:expdecay} does too. To show \eqref{eq:expdecay} we shall use Theorem \ref{thm:Lyapunovbound}. Note that Hypothesis \ref{hyp:upperbound} is satisfied by Lemma \ref{lem:hypcheckforadditivenoiseex}.}
	\begin{figure}
	\centering
	\includegraphics[width=0.7\linewidth]{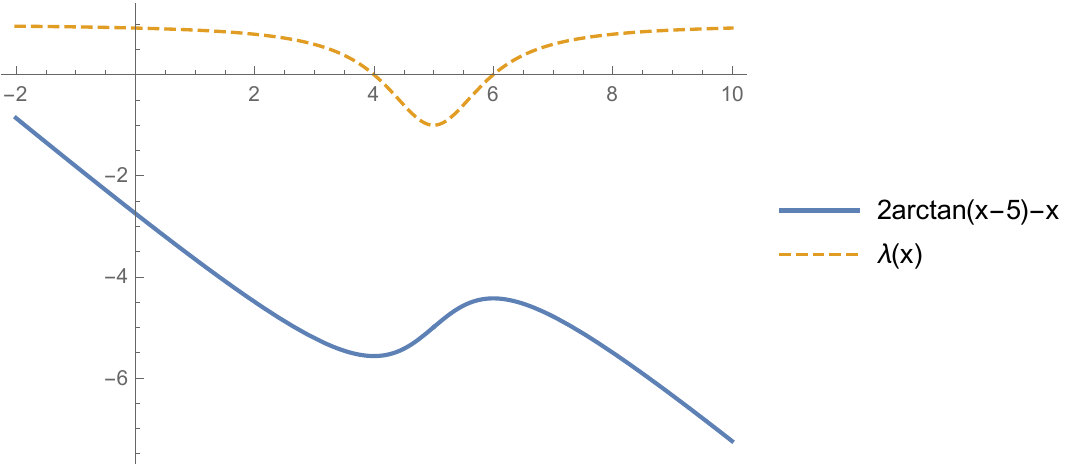}
	\caption{A plot of $V_0(x)$ and $\lambda(x)$ for the SDE \eqref{eq:bump}, see Example \ref{ex:bump}.}
	\label{fig:bplotwithnegbump}
	\end{figure}
	
	\textup{We emphasize that in this example  the function $\lambda$  is bounded below by $-1$ and does take negative values. In Figure \ref{fig:bplotwithnegbump} we plot both $V_0(x)$ and $\lambda(x)$. By Lemma \ref{lem:hypcheckforadditivenoiseex} we have that Hypothesis \ref{hyp:upperbound} is satisfied with $\DV=0.25+0.5(2\arctan(x-5)-x)\tanh(0.5x)$. Then by Theorem \ref{thm:Lyapunovbound} we have that \eqref{eq:fullderivativeestimate} follows provided we can find a $\lambda_0>0$ satisfying \eqref{eq:Vgap}.
	From Figure \ref{fig:vgapwithnegbump} we can see there is a constant $\lambda_0>0$ such that \eqref{eq:Vgap} holds for all $x\in\R$, hence by Theorem \ref{prop:OAC} and Theorem \ref{thm:Lyapunovbound} we have
	\begin{equation*}
	\lvert \partial_x \cP_tf(x) \rvert \leq \cosh(0.5x)e^{-\lambda_0t} \lVert \partial_xf \rVert_\infty.
	\end{equation*}
	The above has been obtained by taking $u(x)=\cosh(0.5x)$,  which we are allowed to do thanks to the proof of Lemma \ref{lem:hypcheckforadditivenoiseex} (with $\alpha=0.5$).}
	\begin{figure}
	\centering
	\includegraphics[width=0.7\linewidth]{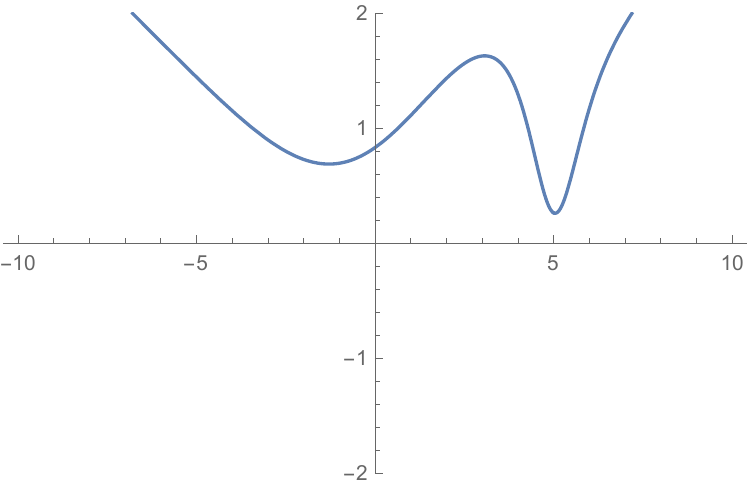}
	\caption{A plot of $2\lambda(x)-\DV(x)$ for the SDE \eqref{eq:bump}, see Example \ref{ex:bump}.}
	\label{fig:vgapwithnegbump}
	\end{figure}
	
	\end{example}

\begin{example}\label{ex:arctan}\textup{
Here we continue Example \ref{ex:arctanprelim}, i.e. we  consider again the SDE \eqref{eq:arctanSDE}. Our goal is to show that the weak error of the Euler approximation of \eqref{eq:arctanSDE} converges to zero uniformly in time; that is, we want to show that \eqref{eq:uniformweakconv} holds. We notice in passing that  this is the case despite the fact that the SDE  \eqref{eq:arctanSDE} does not satisfy the Lyapunov conditions \eqref{eq:Lyapcond} (and it does not satisfy  \eqref{eq:expLyap} for any confining polynomial function $G$, see Note \ref{note:Lyapunov}).}
\textup{To show \eqref{eq:uniformweakconv}, by Theorem \ref{thm:globalerror} it is sufficient to check that Hypothesis \ref{hyp:euler} holds. It is immediate to see that Hypothesis \ref{hyp:euler} \ref{ass:ellipiticity} and \ref{ass:lipschitz} are satisfied.  Hypothesis \ref{hyp:euler} \ref{ass:gradest} is satisfied as well thanks to Lemma \ref{lem:higherorderOACforadditivenoise}. Let us come to explain why this is the case. In the case at hand the only assumption of Lemma \ref{lem:higherorderOACforadditivenoise} which is non-trivial to check is the inequality \eqref{eq:expdecayintermsofb}. Notice that \eqref{eq:expdecayintermsofb} is just \eqref{eq:expdecay} with $\lambda(x)=-b'(x)$, $b(x)$ being the drift in \eqref{eq:arctanSDE}. Therefore, to obtain \eqref{eq:expdecayintermsofb}, we use Theorem \ref{thm:Lyapunovbound}. In turn, to apply Theorem \ref{thm:Lyapunovbound}, we must verify that Hypothesis \ref{hyp:upperbound} holds. This is done in Lemma \ref{lem:hypcheckforadditivenoiseex}, where we show that \eqref{eq:arctanSDE} satisfies Hypothesis \ref{hyp:upperbound} with 
	\begin{equation*}
	\DV(x) = \frac{1}{4}-\frac{1}{2}\arctan(x)\tanh\left(\frac{x}{2}\right).
	\end{equation*}
	From the proof of Lemma \ref{lem:hypcheckforadditivenoiseex} one can moreover see that \eqref{eq:expdecayintermsofb} holds with  $u(x)=\cosh(x/2)$. In Figure \ref{fig:vgapplot} we can see there is a constant\footnote{One can find numerically that $C$ is about $0.267$.} $\lambda_0>0$ such that $2\lambda(x)-\DV(x)\geq 2\lambda_0$ for all $x\in\R$, hence by Theorem \ref{thm:Lyapunovbound} we have 
	\begin{equation*}
	    \mathbb{E}\left[\exp\left(-2\int_0^t \frac{1}{1+(X_s)^2}ds\right)\right] \leq \cosh(x/2)e^{-2\lambda_0t}.
	\end{equation*}
To summarise, Hypothesis \ref{hyp:euler} \ref{ass:gradest} is satisfied with $u(x)=\cosh(x/2)$. As shown in Note \ref{note: applicability},  because the coefficients of this SDE are bounded and have bounded derivatives, verifying Hypothesis \ref{hyp:euler} \ref{hypmom} reduces to showing \eqref{eq:uboundassumption}; this follows from Lemma \ref{lem:expmomentsforarctan}. Finally Hypothesis \ref{hyp:euler} is verified for this example.
	\begin{figure}
		\centering
		\includegraphics[width=0.7\linewidth]{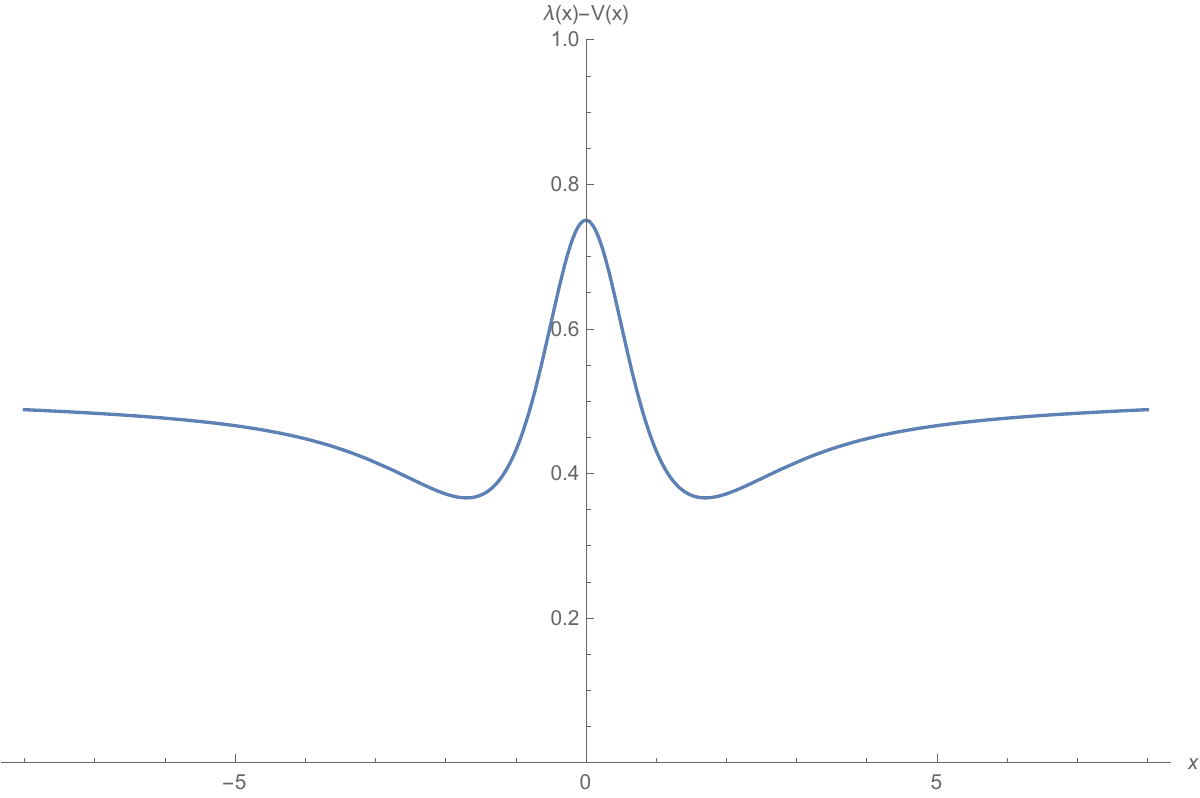}
		\caption{A plot of $2\lambda(x)-\DV(x)$ for the SDE \eqref{eq:arctanSDE}, see Example \ref{ex:arctan}.}
		\label{fig:vgapplot}
	\end{figure}
}
	
\textup{	Note that another consequence of \eqref{eq:expdecayuptoorder4nonuniform} is that the SDE \eqref{eq:arctanSDE} decays to equilibrium exponentially fast. One can check directly that \eqref{eq:arctanSDE} admits an invariant measure and such an invariant measure has a density with respect to the Lebesgue measure on $\R$ given by
	\begin{equation*}
	    \mu(x) = \frac{1}{Z} \sqrt{1+x^2}e^{-x\arctan(x)},
	\end{equation*}
	where $Z$ is a normalising constant. Then for $f\in C_b^1(\R)$ we have
	\begin{align*}
	    \left\lvert\cP_tf(x) -\int_\R f(y)\mu(y)dy\right\rvert &= \left\lvert\cP_tf(x) -\int_\R \cP_tf(y)\mu(y)dy \right\rvert\\
	    &=\left\lvert\int_\R(\cP_tf(x) - \cP_tf(y))\mu(y)dy\right\rvert\\
	    &=\left\lvert\int_\R \int_{x}^y \partial_z\cP_tf(z) dz \mu(y)dy\right\rvert\\
	    &\leq \lVert f'\rVert_\infty e^{-\lambda_0 t} \int_\R \int_{x}^y \cosh(z/2) dz \mu(y)dy\\
	    &\leq K(x)\lVert f'\rVert_\infty e^{-\lambda_0t} .
	\end{align*}
	Here
	$$
	K(x) = \int_\R \int_{x}^y \cosh(z/2) dz \mu(y)dy
	$$
	which is finite for all $x\in \R$.}\hf
\end{example}

Now we consider an example which does not satisfy Hypothesis \ref{hyp:euler} and we can see numerically that for this example the conclusion of Theorem \ref{thm:globalerror} does not hold.

\begin{example}\label{ex:Grusin}
\textup{Consider the two dimensional SDE
\begin{equation}\label{eq:Grusin}
    \begin{cases}
    dX_t^1&=X_t^1dt\\
    dX_t^2&=\sqrt{2}X_t^1\circ dB_t.
    \end{cases}
\end{equation}
Here $V_0= x^1\pa_{x^1}$, $V_1=x^1 \pa_{x^2}$. It is shown in \cite[Example 6.9]{CCDO} that for this example the Obtuse Angle Condition \eqref{OAC1d} is satisfied (by $V=V_1$) with $\lambda=1$ and therefore the derivatives of the semigroup decay exponentially fast. However  the moment  bounds \eqref{eq:uboundinexp}-\eqref{eq:uboundinexp3} on the Euler approximation of \eqref{eq:Grusin} do not hold true. Indeed,  both the second component of the stochastic process $X_t^2$ and the second component of the Euler approximation $Y_{t_{n(t)}}^{\delta,2}$ are distributed according to a Gaussian random variable and one can show
\begin{align*}
    \mathrm{Var}(X_t^2) &= e^{2t}-1\\
    \mathrm{Var}(Y_{t_{n(t)}}^{\delta,2}) &= 2\frac{(1 + \delta)^{2 n} - 1}{2 + \delta}.
\end{align*}}

\textup{Moreover, in Figure \ref{fig:differenceinvar} we see that as $t$ tends to $\infty$ the difference between the variance of $X_t^2$ and $Y_{t_{n(t)}}^{\delta,2}$ diverges. In particular, this implies that the Euler Approximation does not weakly converge uniformly in time.}

\textup{\begin{figure}
    \centering
    \includegraphics[width=0.8\linewidth]{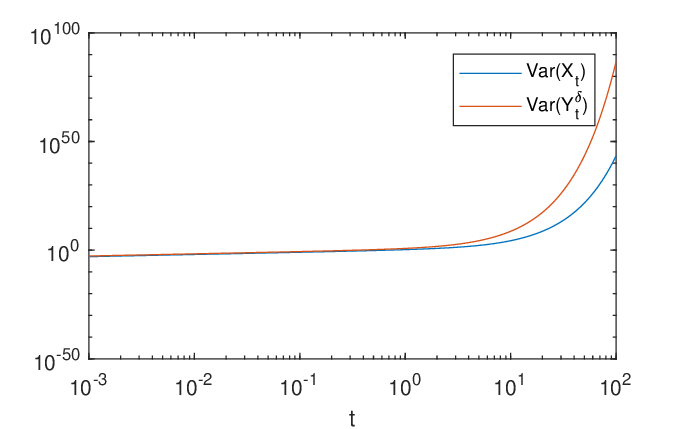}
    \caption{This figure is a plot of $\mathrm{Var}(X_t^2)$ and $\mathrm{Var}(Y_{t_{n(t)}}^{\delta,2})$ when $\delta=10^{-3}$, for the SDE \eqref{eq:Grusin}.}
    \label{fig:differenceinvar}
\end{figure}}
\hf
\end{example}

\begin{remark}\label{note:Lyapunov}
\textup{Here we make some comments on the relation between i) decay of derivatives of the semigroup (i.e. estimates of the type \eqref{eq:expdecayuptoorder4nonuniform} or \eqref{expdecayintro}); ii) uniform in time moment bounds for the Euler approximation (i.e. bounds of the type \eqref{eq:uboundinexp}-\eqref{eq:uboundinexp3}); iii) uniform in time convergence of the Euler approximation (i.e. \eqref{stronguniform}); and Lyapunov conditions of the type \eqref{eq:Lyapcond} or \eqref{eq:expLyap}. 
\begin{itemize}
    \item In \cite{Mattingly} the authors show that Lyapunov conditions of the type \eqref{eq:expLyap} are not robust under discretization,  and indeed ergodicity may be lost after discretising, see \cite[Section 6]{Mattingly} and references therein for a complete discussion. In Example \ref{ex:xcubed} we exhibit a simple one dimensional SDE (similar to the one presented in \cite[Section 6]{Mattingly}) which does satisfy  \eqref{eq:expLyap} and the property i); however the bounds ii) and the uniform weak convergence iii) only hold if the step-size is chosen to be small enough (the ``smallness" is determined by the size of the initial datum). 
    \item In \cite[Section 3.1]{TalayTubaro} the authors show that, in the case of elliptic SDEs, the Lyapunov condition \eqref{eq:Lyapcond} implies the bounds ii). Here we exhibit an example, Example \ref{ex:circlewithconfinement}, where \eqref{eq:Lyapcond} is satisfied and (despite the fact that the noise is degenerate) ii) does hold. However \eqref{stronguniform} does not. 
    \item Finally, the SDE in Example \ref{ex:arctan} does not satisfy \eqref{eq:Lyapcond} and it does not satisfy \eqref{eq:expLyap} for any confining \footnote{The function $G$ is said to be confining if $G(x) \rightarrow \infty$ when $\lv x \rv \rightarrow \infty$. } function $G$ with polynomial growth; however Theorem \ref{thm:globalerror} applies to such a dynamics and therefore \eqref{stronguniform} does hold.  We note that while \eqref{eq:expLyap} does not hold for any even polynomial function $G$, it does hold for $G= \cosh(x)$, see \eqref{eq:Lyapunovcondforarctan}.
\end{itemize}
}
\end{remark}

\begin{example}\label{ex:xcubed}
\textup{Consider the one-dimensional SDE
\begin{equation}\label{eqn:xcubed}
dX_t = (-X_t^3-X_t) dt + \sqrt{2} dB_t.
\end{equation}
Let us start by observing that the function $G(x)=1+x^2$ is a Lyapunov function for such an SDE in the sense that, if $\cL$ is the generator of \eqref{eqn:xcubed}, then one has 
\begin{equation}\label{eq:expLyap}
(\cL G) (x)\leq - c G(x)+ d
\end{equation}
for some $c,d>0$ (with a calculation completely analogous to the one in \cite[equation (6.9)]{Mattingly}). }
\textup{Moreover, a straightforward calculation shows that the second moment of $X_t$ is bounded uniformly in time, i.e.
$$
\mathbb{E}\lv X_t\rv^2 \leq C, 
$$
for some constant $C$ independent of time. However the same is not true for the corresponding Euler approximation. More precisely, 
by following the same argument as in the proof of \cite[Lemma 6.3]{Mattingly} one can show the following: 
\be\label{eq:Mattinglysecondmoment}
\mbox{if } \, \, \mathbb{E}[(Y_0^\delta)]^2\geq 4+4/\delta^2 \mbox{ then } \mathbb{E}[(Y_{t_n}^\delta)^2] \rightarrow \infty \mbox{ as } n\rightarrow\infty \,.
\ee
In other words, if we fix a large initial datum then, in order for the second moment to stay bounded we need to choose a sufficiently small step-size. Therefore, if the initial datum is not small enough, bounds of the type \eqref{eq:uboundinexp}-\eqref{eq:uboundassumption} cannot hold.
One can also show with a slightly lengthy but simple calculation\footnote{This calculation follows the scheme outlined in Note \ref{compBakry}, i.e. define the function $\Gamma(f):= \sum_{k=1}^4\lv V_1^{(k)}f\rv^2$ and then prove that \eqref{eq:strategystandard} holds for such a function.}
 that  the derivatives (up to order four) of the semigroup generated by \eqref{eqn:xcubed} decay exponentially fast, namely
$$
\sum_{k=1}^4\lv V_1^{(k)}\cP_tf(x)\rv^2 \leq c e^{-ct},
$$
for some constant $c>0$. We emphasize that here $V_1=\pa_x$ so the above estimates are actually derivative estimates in the coordinate direction.  The plots in Figure \ref{fig:xcubedconvergent} and Figure \ref{fig:xcubednonconvergent} then show that, for a fixed initial datum, if the step-size $\delta$ is small enough then \eqref{stronguniform} holds, otherwise it doesn't (coherently with \eqref{eq:Mattinglysecondmoment}).}

\begin{figure}
    \centering
    \includegraphics[width=0.8\linewidth]{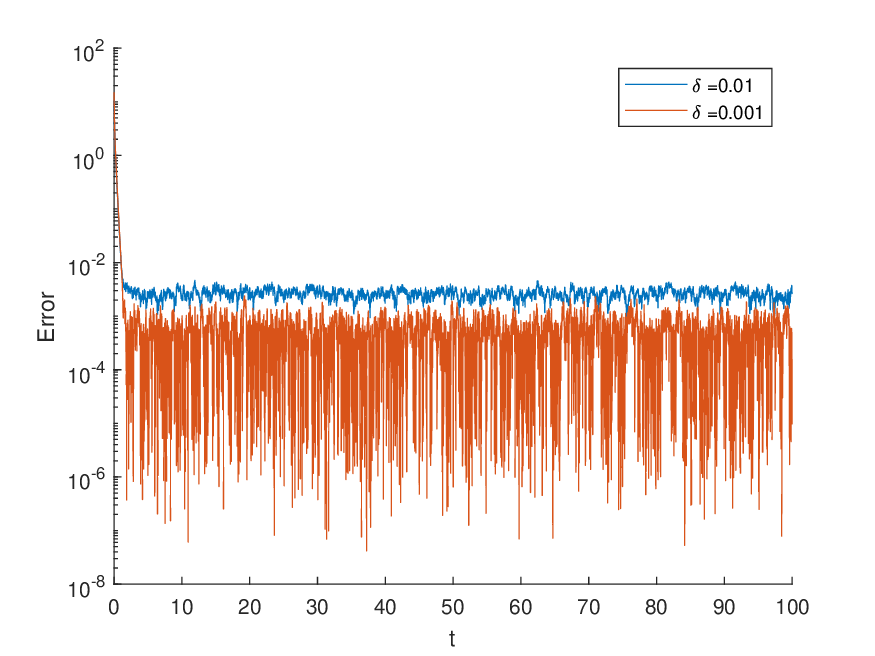}
    \caption{This figure is a plot of the error between $\mathbb{E}[\lvert Y_{t}^\delta\rvert^2]$ and $\int y^2 \mu(dy)$ for the SDE \eqref{eqn:xcubed}. Here we averaged over $10^6$ simulations to estimate the expectation and used the initial condition $X_0=Y_0^\delta=4$. Here $\mu$ is the unique invariant measure of \eqref{eqn:xcubed}. See Example \ref{ex:xcubed} for details.}
    \label{fig:xcubedconvergent}
\end{figure}
\begin{figure}
    \centering
    \includegraphics[width=0.8\linewidth]{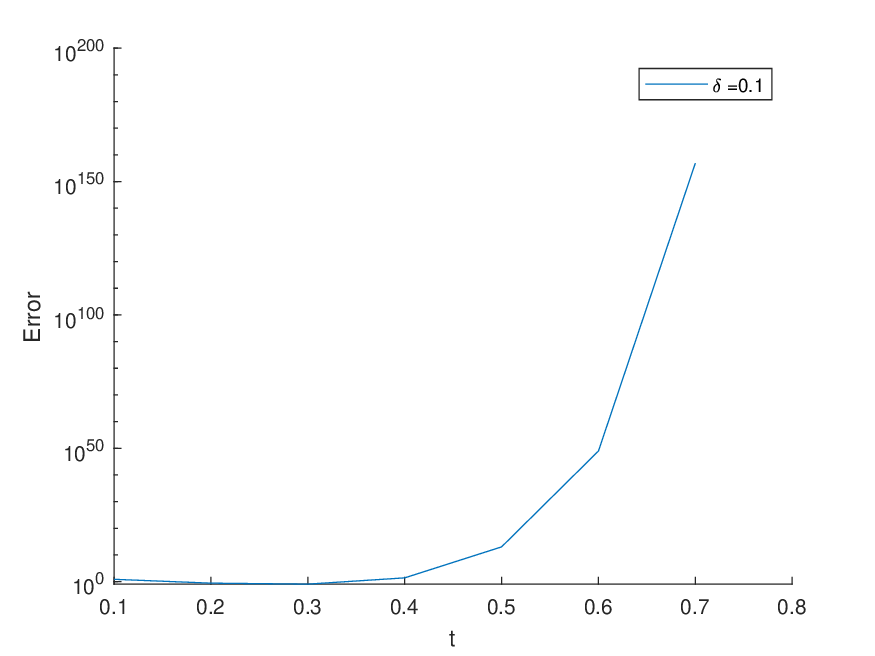}
    \caption{This figure is a plot of the error between $\mathbb{E}[\lvert Y_{t}^\delta\rvert^2]$ and $\int y^2 \mu(dy)$ for the SDE \eqref{eqn:xcubed}. Here we averaged over $10^6$ simulations to estimate the expectation and used the initial condition $X_0=Y_0^\delta=4$. Here $\mu$ is the unique invariant measure of \eqref{eqn:xcubed}. See Example \ref{ex:xcubed} for details.}
    \label{fig:xcubednonconvergent}
\end{figure}
\hf
\end{example}

\begin{example}\label{ex:circlewithconfinement}
	\textup{Consider the two dimensional ODE
	\begin{equation}\label{eq:circlewithconfinement}
	\begin{cases}
	\frac{d}{dt}X_t^1 = \left( -X_t^2 +\Psi(X_t)X_t^1\right) dt\\
	\frac{d}{dt}X_t^2 =\left(  X_t^1 +\Psi(X_t)X_t^2\right) dt
	\end{cases}
	\end{equation}
	where $\Psi:\R^2\to\R$ is a smooth bounded function such that $\Psi(x)=0$ if $\lvert x\rvert<2$ and $\Psi(x)=-1$ if $\lvert x\rvert>3$.  This dynamics provides an example where  Hypothesis \ref{hyp:euler} \ref{hypmom} is satisfied (at least when $\delta<1$) while Hypothesis \ref{hyp:euler}   \ref{ass:gradest} is not; moreover, the conclusion \eqref{eq:uniformweakconv} of Theorem \ref{thm:globalerror} does not hold, i.e. in this case the weak error of the Euler approximation does not converge to zero uniformly in time. At the end of this example we will also add (degenerate) noise to the above dynamics and show that the same reasoning still applies, see below. Before moving on to looking at this example in more detail, we would also like to emphasize that the dynamics \eqref{eq:circlewithconfinement} does satisfy a Lyapunov-type condition; indeed, if $b(x)$ is the drift of the equation, then outside of the ball of radius three one has 
	\be\label{eq:Lyapcond}
	x \cdot b(x) \leq - \left \vert x \right \vert^2 \,.
	\ee
	}
	
	\textup{To see that Hypothesis \ref{hyp:euler} \ref{ass:gradest} does not hold, fix some $x=(x^1, x^2)$ with $\lvert x \rvert <2$; then we may solve \eqref{eq:circlewithconfinement} to find
	\begin{align*}
	X_t^1&= x^1 \cos(t)-x^2\sin(t)\\
	X_t^2&= x^1\sin(t)+x^2\cos(t).
	\end{align*}
	For $f\in C_b^\infty(\R^2)$ we then have
	\begin{align*}
	\partial_1\cP_tf(x) = \cos(t)(\partial_1f)(X_t)+\sin(t)(\partial_2f)(X_t).
	\end{align*}
From the right hand side of the above expression we see that $\partial_1\cP_tf(x)$ will not converge to zero as $t$ tends to $\infty$ for all $f\in C_b^\infty(\R^2)$.}

\textup{In order to prove that Hypothesis \ref{hyp:euler} \ref{hypmom} holds, we shall show that $\lvert Y_{t_n}^\delta\rvert^2$ is bounded independently of $n$. Let $R_n:=\lvert Y_{t_n}^\delta\rvert^2$, note that $R_n$ satisfies the recurrence relation
\begin{equation}\label{eq:circleradiusreccurance}
R_{n+1} = ((1+\Psi(Y_{t_n}^\delta)\delta)^2+\delta^2)R_n.
\end{equation}
Suppose for some $n$ that $R_n>3$, in which case $\Psi(Y_{t_n}^\delta)=-1$ and \eqref{eq:circleradiusreccurance} can be rewritten as
\begin{equation*}
R_{n+1} = (1-2\delta+2\delta^2)R_n.
\end{equation*}
Therefore, provided $\delta<1$,  we see that $R_{n+1}<R_n$ which implies that $R_n$ is bounded independent of $n$.}

\textup{From \eqref{eq:circleradiusreccurance} we also see that if we take the initial condition to be $x=(1,0)$ then $X_t$ will remain on the circle of radius $1$ whereas $R_n$ will increase towards 2 and then remain in a small region around 2 from then on. Hence 
\begin{equation*}
\sup_{n}( \lvert Y_{t_n}^{\delta}\rvert^2-\lvert X_{t_n}\rvert^2) >1.
\end{equation*}
This is also demonstrated in Figure \ref{fig:circleconfinement} for three choices of $\delta$ (for this figure we took $\Psi\in C_b^4(\R^2)$ as described in \eqref{eq:circlewithconfinement} and defined by a polynomial interpolation for $2<\lvert x\rvert<3$).}

\textup{Let us now add noise to the  ODE \eqref{eq:circlewithconfinement} and consider the system
\begin{equation}\label{eq:circlewithconfinementnoise}
	\begin{cases}
	\frac{d}{dt}X_t^1 = \left( -X_t^2 +\Phi(X_t)X_t^1\right) dt + \mathbf{1}_{\left\vert X_t \right \vert> 3} dW_t^1\\
	\frac{d}{dt}X_t^2 =\left(  X_t^1 +\Phi(X_t)X_t^2\right) dt+\mathbf{1}_{\left\vert X_t \right \vert> 3} dW_t^2 \, ,
	\end{cases}
	\end{equation}
where $W_t^1, W_t^2$ are one-dimensional independent Brownian motions. Then again the space derivatives of the semigroup do not decay to zero and the Euler approximation remains bounded, but it will not approximate the SDE uniformly in time, see Figure \ref{fig:circleconfinementnoise}. }

\textup{\begin{figure}
    \centering
    \includegraphics[width=0.8\linewidth]{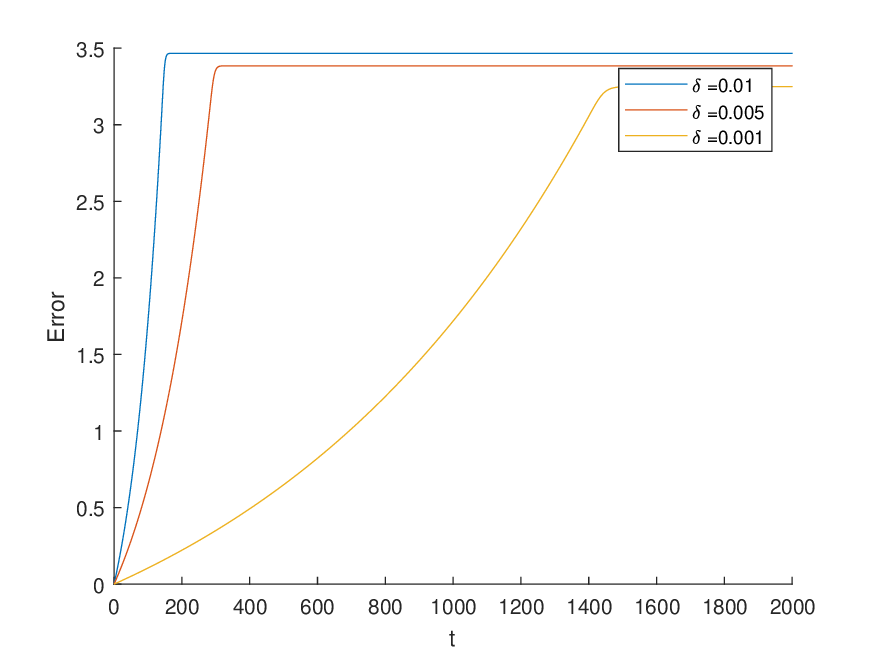}
    \caption{This figure is a plot of the error between $\lvert X_t\rvert^2$ and  $\lvert Y_t^\delta\rvert^2$, where $X_t$ is the solution of the ODE \eqref{eq:circlewithconfinement} and $Y_t^\delta$ its Euler approximation, for various choices of the step-size $\delta$. As $\delta$ tends to zero the error does not tend to zero, hence $\eqref{stronguniform}$ cannot hold.}
    \label{fig:circleconfinement}
\end{figure}}

\textup{\begin{figure}
    \centering
    \includegraphics[width=0.8\linewidth]{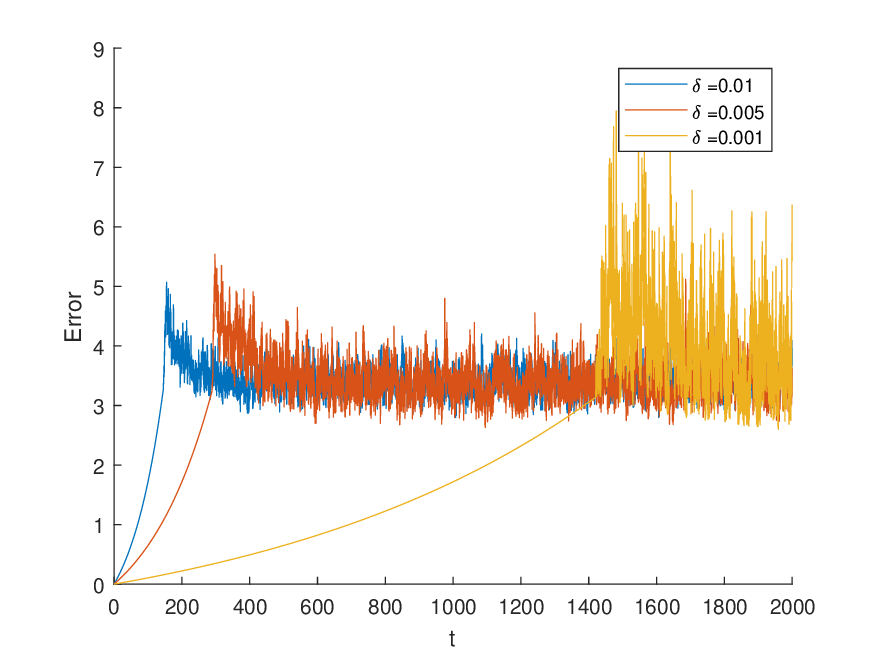}
    \caption{This figure is a plot of the error between $\mathbb{E}[\lvert X_t\rvert^2]$ and $\mathbb{E}[\lvert Y_t^\delta\rvert^2]$ for the SDE \eqref{eq:circlewithconfinementnoise}. Here we averaged over $10000$ simulations to estimate the expectation and used the initial condition $X_0=Y_0^\delta=(1,0)$. Similarly to what happens in Figure \ref{fig:circleconfinement}, as $\delta$ tends to zero the error is not tending to zero, hence \eqref{stronguniform} cannot hold. }
    \label{fig:circleconfinementnoise}
\end{figure}}
\hf
\end{example}

{\bf Acknowledgments}. P. Dobson was supported by the Maxwell Institute Graduate School in Analysis and its
Applications (MIGSAA), a Centre for Doctoral Training funded by the UK Engineering and Physical
Sciences Research Council (grant EP/L016508/01), the Scottish Funding Council, Heriot--Watt
University and the University of Edinburgh. The work of D. Crisan has been partially supported by a UC3M-Santander Chair of Excellence grant held at the Universidad Carlos III de Madrid.

\appendix
\numberwithin{equation}{section}

\section{UFG condition}\label{app:UFG}

Here we briefly gather some background material on the UFG condition, which was first introduced by Kusuoka and Stroock in \cite{{KusStr82},{KusStr85},{KusStr87}, Kus03} and later further studied by Crisan and collaborators in \cite{{CrisanLitterer}, {CrisanDelarue}, Crisan}, in particular they show that the UFG condition is a sufficient condition to ensure the semigroup $\cP_tf$ is smooth in the directions $V_{[\alpha]}$, which are defined below.

Fix $d \in \N$ and let $\A$ be the set of all $k$-tuples, of any size $k \geq 1$,  of integers of the following form 
$$
\A:= \{\alpha=(\alpha^1 \dd \alpha^{k}), k \in \N : \alpha^j\in \{0, 1 \dd d\} {\mbox{ for all } j \geq  1} \}\setminus \{(0)\}\,.
$$
We emphasise that all $k$-tuples of any length $k\geq 1$ are allowed in $\A$,  except the trivial one, $\alpha=(0)$ (however singletons $\alpha=(j)$ belongs to $\A$ if $j\in\{1\dd d\}$). 
We endow $\A$ with the product operation
$$
\alpha \ast \beta:= (\alpha^1 \dd \alpha^h, \beta^1 \dd \beta^{\ell}), 
$$
for any $\alpha=(\alpha^1 \dd \alpha^h)$ and 
$\beta=(\beta^1 \dd \beta^{\ell})$ in $\A$.  If  $\alpha\in\A$, we define the {\em length} of $\alpha$, denoted by $\|\alpha\|$, to be the integer
$$
\|\alpha\|:= h+\mbox{card}\{i: \alpha_i=0\} , \qquad \mbox{if } \alpha=(\alpha^1 \dd \alpha^h) \,.
$$
For any $m \in \N, m\geq 1$, we then  introduce the sets
\begin{align*}
&\A_m=\{\alpha \in \A: \| \alpha\|\leq m   \} \,.
\end{align*}

Let $\{V_i: i=0 \dd d \}$ be a collection of vector fields on $\R^N$ and let us define the following ``hierarchy" of operators:
\begin{align*}
V_{[i]} &:= V_i \qquad i=0, 1 \dd d\\
V_{[\alpha \ast i]} & := [V_{[\alpha]}, V_{[i]}], \qquad \alpha \in \A, i=0,1 \dd d\,.
\end{align*}
Note that if $\| \alpha\|=h$ then $\| \alpha \ast i \| = h+1$ if $i \in \{1 \dd d\}$ and $\| \alpha \ast i\|= h+2$ if $i=0$. 
Let $C^{\infty}_V(\R^N)$ denote the set  of  bounded smooth functions,  $\varphi: \R^N \rightarrow \R$,   such that
\begin{equation*}\label{supphi}
\sup_{x \in \R^N}\lv V_{[\gamma_{1}]} \dots V_{[\gamma_{k}]} \varphi \rv < \infty
\end{equation*}
for all $k$ and all $\gamma_{1} \dd \gamma_{k} \in \A_m$. 
With this notation in place we can now introduce the UFG condition.

\begin{definition}[UFG Condition]\label{defufg}
Let $\{V_i: i=0 \dd d \}$ be a collection of 
smooth vector fields on $\R^N$ and assume that the coefficients of such vector fields have bounded partial derivatives (of any order). We say that the vector fields  $\{V_i: i=0 \dd d \}$ satisfy the UFG condition if there exists $m\in\N$ such that for any $\alpha \in \A$ of the form 
$$
\alpha = \alpha'\ast i, \qquad    \alpha' \in \A_m, \, i \in \{0 \dd d\}, $$ 
 one can find  bounded smooth functions $\varphi_{\alpha, \beta}=\varphi_{\alpha, \beta}(x) \in C^{\infty}_V(\R^N)$ such that
\begin{equation*}\label{eq:UFG}
V_{[\alpha]}(x) = \sum_{\beta \in \A_m} \varphi_{\alpha,\beta}(x) V_{[\beta]} (x) \,.
\end{equation*}
\end{definition}

For our purposes, it is important to stress that any ellipitic process satisfies the UFG condition and analogously, any uniformly hypoellipitic processes is of UFG type as well, see \cite{CCDO}. We can define a version of the Obtuse Angle Condition for UFG processes. Indeed in \cite{CrisanOttobre} it is shown that if the \emph{Obtuse Angle Condition} is satisfied by all the vectors $V_{[\alpha]}$, i.e. if
\begin{equation}\label{eq:OACintro}
([V_{[\alpha]},V_0]f)(V_{[\alpha]}f) \leq -\lambda_0 \lvert V_{[\alpha]}f\rvert^2, \quad \forall \alpha\in\A_m, f \text{ sufficiently smooth} 
\end{equation}
then
\begin{equation*}
    \lvert V_{[\alpha]}\cP_tf(x)\rvert \leq C e^{-\overline{\lambda}t} \sum_{\beta\in\A_m} \lVert V_{[\beta]} f\rVert_\infty
\end{equation*}
for some positive constants $\overline{\lambda},C$ and for any $f$ sufficiently smooth, $\alpha\in\A_m$ and $x\in \R^N$, see \cite{CrisanOttobre} for details.

 Observe that we can equivalently\footnote{Note that we can write \eqref{eq:OACintro} as
\begin{equation*}
    \nabla f(x)^T [V_{[\alpha]},V_0](x) V_{[\alpha]}(x)^T\nabla f(x) \leq -\lambda_0 \lvert \nabla f(x)^T V_{[\alpha]}(x)\rvert^2, \quad \forall \alpha\in\A_m, f \text{ sufficiently smooth}.
\end{equation*}
Fix $x,\xi\in\R^N$ then by taking $f\in C_c^\infty(\R^N)$ with $f(y)=y^T\xi$ in some neighbourhood of $x$ we obtain \eqref{eq:OACvector}.
}
 express \eqref{eq:OACintro} as 
\begin{equation}\label{eq:OACvector}
    \xi^T [V_{[\alpha]},V_0](x) V_{[\alpha]}(x)^T\xi \leq -\lambda_0 \lvert \xi^T V_{[\alpha]}(x)\rvert^2, \quad \forall x,\xi\in\R^n, \alpha\in\A_m.
\end{equation}

At this level of generality, the Local Obtuse Angle Condition \eqref{LOAC1d} which we previously stated just for the case $d=N=1$, becomes the requirement that there is some measurable function $\lambda:\R^N \to \R$ such that for all $f$ sufficiently smooth
\begin{equation}\label{eq:LOAC}
    ([V_{[\alpha]},V_0]f)(x)(V_{[\alpha]}f)(x) \leq -\lambda(x) \lvert V_{[\alpha]}f(x)\rvert^2, \quad \forall x\in\R^N, \alpha\in\A_m.
\end{equation}
\section{Auxiliary proofs}\label{app:auxproofs}

\begin{lemma}\label{lem:OACforaddnoise}
    Consider the following SDE in $\R^N$
\begin{equation*}\label{eq:SDEaddnoiceintro}
dX_t=V_0(X_t)dt+\sqrt{2}\sum_{k=1}^N e_idB_t^i, \quad X_0=x,
\end{equation*}
where $\{e_i\}$ are the canonical basis vectors of $\R^N$. If the Obtuse Angle Condition \eqref{eq:OACvector} holds for the above SDE then $V_0$ is unbounded and $X^i_t$ is independent of $X^j_t$ for each $t>0$ and $i\neq j$.
\end{lemma}

\begin{proof}[Proof of Lemma \ref{lem:OACforaddnoise}]
In this case the OAC \eqref{eq:OACvector} becomes
\begin{equation*}
\sum_{j=1}^N\partial_iV_0^j(x)\xi^i\xi^j \leq -\lambda_0 \lvert \xi^i\rvert^2, \quad \forall x,\xi\in\R^N, i\in\{1,\ldots, N\}.
\end{equation*}
Fix some $i\in\{1,\ldots,N\}$ and take $\xi=e_i$; then we have
\begin{equation*}
\partial_iV_0^i(x) \leq -\lambda_0, \quad \forall x\in\R^N.
\end{equation*}
Integrating the above gives $V_0^i(x)\leq V_0^i(x_1,\ldots,x_{i-1}, 0,x_{i+1},\ldots,x_N)-\lambda_0 x_i$ for $x_i>0$ and $i\in\{1,\ldots,N\}$. Now letting $x_i$ tend to $\infty$ we must have that $V_0(x)$ is unbounded below.

Moreover, if we take $\xi =e_i+Ke_k$ for some $i\neq k$ then by \eqref{eq:OACvector} we have
\begin{equation*}
\partial_iV_0^i+K\partial_iV_0^k(x)\leq -\lambda_0 , \quad \forall x\in\R^N, i\in\{1,\ldots,N\}, k\neq i, K\in\R.
\end{equation*}
By considering both the cases when $K$ is large and negative, and when $K$ is large and positive we must have that $\partial_iV_0^k=0$ for $k\neq i$. Therefore $X_t^i$ is independent of $X_t^j$ for $i\neq j$. 
\end{proof}

\begin{lemma}\label{lem:simplificationin1d}
	Consider the SDE \eqref{eq:SDEito} when $N=1$, i.e. consider the SDE
	\begin{equation}\label{eq:onedimSDE}
dX_t=U_0(X_t)dt+\sqrt{2}\sum_{k=1}^dV_k(X_t)dB_t^k.
\end{equation}
	Then we may find a vector field $U_1$ such that $X_t$ is a weak solution to the SDE
	\begin{equation}\label{eq:onedimSDEito}
	dX_t=U_0(X_t)dt+\sqrt{2}U_1(X_t)dW_t
	\end{equation}
	for some one-dimensional Brownian motion $\{W_t\}_{t\geq 0}$. Moreover, if the Local Obtuse Angle Condition \eqref{eq:LOAC} is satisfied by the vector fields in \eqref{eq:onedimSDE}, then we have
	\begin{equation*}
	U_1(x)[U_1,V_0](x) \leq -\lambda(x) U_1(x)^2,\quad \text{ for every } x\in\R,
	\end{equation*}
	where $V_0$ is defined by \eqref{eq:Itodrift}.
\end{lemma}

\begin{proof}
	Define the process
	\begin{equation*}
	W_t=\sum_{i=1}^d\int_0^t \frac{V_i(X_s^{x})}{\sqrt{\sum_{j=1}^d\lvert V_j(X_s^{(x)})\rvert^2}} dB_s^i.
	\end{equation*}
	By the Levy Characterisation of Brownian motion (see \cite[Theorem 3.3.16]{KaratzasShreve}), $W_t$ is a one-dimensional Brownian motion. With this in mind, we have
	\begin{align*}
	dX_t^{(x)} &= U_0(X_t^{(x)})dt + \sqrt{2}\sum_{i=1}^d V_i(X_t^{(x)}) dB_t^i\\
	&=U_0(X_t^{(x)})dt + \sqrt{2}\left(\sum_{j=1}^d\lvert V_j(X_t^{(x)})\rvert^2\right)^\frac{1}{2} dW_t\\
	&=U_0(X_t^{(x)})dt + \sqrt{2}U_1(X_t^{(x)}) dW_t,
	\end{align*}
	where we set
	\begin{equation*}
	U_1(x) := \left(\sum_{j=1}^d\lvert V_j(x)\rvert^2\right)^{\frac{1}{2}}.
	\end{equation*}
	Since
	\begin{equation*}
	U_1(x)U_1'(x) = U_1(x) \sum_{j=1}^d V_j(x)V_j'(x)\left(\sum_{i=1}^d\lvert V_i(x)\rvert^2\right)^{-\frac{1}{2}} = \sum_{j=1}^d V_j(x)V_j'(x)
	\end{equation*}
	we have that $X_t^{(x)}$ satisfies the Stratonovich SDE
	\begin{equation*}
	dX_t^{(x)}= V_0(X_t^{(x)})dt +\sqrt{2}U_1(X_t^{(x)})\circ dW_t.
	\end{equation*}
	
	Note that
	\begin{align*}
	U_1(x)[U_1,V_0](x) &= U_1(x)^2V_0'(x)-U_1(x)U_1'(x)V_0(x)\\
	&=\sum_{j=1}^d\lvert V_j(x)\rvert^2V_0'(x)-V_j(x)V_j'(x)V_0(x)\\
	&=\sum_{j=1}^d V_j(x)[V_j,V_0](x).
	\end{align*}
	Therefore, if \eqref{eq:LOAC} is satisfied, we have
	\begin{equation}\label{eq:OACforU}
	U_1(x)[U_1,V_0](x) \leq -\lambda(x) \sum_{j=1}^d \lvert V_j(x)\rvert^2=-\lambda(x) \lvert U_1(x)\rvert^2.
	\end{equation}
\end{proof}

Note that the transformation in Lemma \ref{lem:simplificationin1d} does not necessarily preserve the UFG condition however it will preserve a local version of it, the LFG which we recall below.
\begin{definition}[LFG Condition]
Let $\{V_i: i=0 \dd d \}$ be a collection of 
smooth vector fields on $\R^N$ and assume that the coefficients of such vector fields have bounded partial derivatives (of any order). We say that the vector fields  $\{V_i: i=0 \dd d \}$ satisfy the LFG condition if for each $x\in \R$ there exists some neighbourhood $\mathcal{O}$ of $x$, and some $m\in\N$ such that for any $\alpha \in \A$ of the form 
$$
\alpha = \alpha'\ast i, \qquad    \alpha' \in \A_m, \, i \in \{0 \dd d\}, $$ 
 one can find  smooth functions $\varphi_{\alpha, \beta}=\varphi_{\alpha, \beta}(x) \in C^{\infty}_V(\R^N)$ such that
\begin{equation}\label{eq:LFG}
V_{[\alpha]}(y) = \sum_{\beta \in \A_m} \varphi_{\alpha,\beta}(y) V_{[\beta]} (y) \quad \forall y\in \mathcal{O}.
\end{equation}
\end{definition}

Let us now recall that a one dimensional SDE with multiplicative noise can be recast into a (one-dimensional) SDE with additive noise by using a Lamperti transformation, see \cite[Section 5.2.C]{KaratzasShreve}, assuming the coefficients of the initial SDE are bounded and satisfy an ellipticity condition. 

\begin{lemma}\label{lem:lamperti}
    Consider a one-dimensional SDE with multiplicative noise of the form \eqref{eq:onedimSDEito} and suppose the vector field $U_1$ appearing in \eqref{eq:onedimSDEito} is such that \eqref{eq:onedimSDEito} is uniformly elliptic. Then we can construct a smooth diffeomorphism $h$ such that $Y_t:=h(X_t)$  is the solution to 
    \begin{equation}\label{eq:transformedSDE}
dY_t=b_Y(Y_t)dt+\sqrt{2}dB_t
\end{equation}
    for some smooth function $b_Y$. Moreover, \eqref{eq:onedimSDEito} satisfies the Obtuse Angle condition \eqref{OAC1d} with constant $\lambda_0$ if and only if $b_Y'\leq -\lambda_0$.
\end{lemma}

\begin{proof}[Proof of Lemma \ref{lem:lamperti}] 
Consider the one dimensional SDE in It\^o form \eqref{eq:onedimSDEito}. By the uniform ellipticity assumption there is some constant $\nu>0$ such that $U_1(x)\geq \nu$ for all $x\in\R$. Fix some arbitrary $x_0\in\R$ and define the function $h$ as follows
\begin{equation*}
h(x) = \int_{x_0}^x \frac{1}{U_1(y)}dy.
\end{equation*}
Let $Y_t=h(X_t)$, then $Y_t$ is a strong solution of the SDE \eqref{eq:transformedSDE} where
\begin{equation*}
b_Y(y) = \frac{U_0(h^{-1}(y))}{U_1(h^{-1}(y))}-U_1'(h^{-1}(y))\stackrel{\eqref{eq:Itodrift}}{=} \frac{V_0(h^{-1}(y))}{U_1(h^{-1}(y))}.
\end{equation*}

The derivative of $b_Y$ is given by
\begin{align*}
b_Y'(y) & =\frac{V_0'(h^{-1}(y))U_1(h^{-1}(y))-U_1'(h^{-1}(y))V_0(h^{-1}(y))}{U_1(h^{-1}(y))^2} \frac{d}{dy}h^{-1}(y)\\
&= \frac{[U_1,V_0](h^{-1}(y)) U_1(h^{-1}(y))}{U_1(h^{-1}(y))^2}.
\end{align*}
From the above the statement follows.
\end{proof}

\begin{lemma}\label{lem:onlyneedlevel1}
Consider the one dimensional SDE \eqref{eq:SDEmain1d}. If the UFG condition (see Appendix \ref{app:UFG}) and the LOAC \eqref{eq:LOACforV1} hold then  for all $x\in\R$, 
\begin{equation}\label{eq:equalspans}
\mathrm{span}(V_1(x)) = \mathrm{span}(V_{[\alpha]}(x): \alpha\in\A_m).
\end{equation}
\end{lemma}

\begin{proof}[Proof of Lemma \ref{lem:onlyneedlevel1}]
Fix some $x_0\in\R$; if $V_1(x_0) \neq 0$ then we have $\mathrm{span}(V_1(x_0))=\R$ and \eqref{eq:equalspans} follows immediately. 

If $V_1(x_0)= 0$, we want to prove by induction that   $V_{[\alpha]}(x_0)=0$ for every $\alpha\in\A_m$. To this end, suppose $V_{[\alpha]}(x_0)=0$ for some $\alpha \in \A_m$;   then we may use Taylor's theorem to obtain the following expansions
\begin{align*}
    V_{[\alpha]}(x) &= V_{[\alpha]}'(x_0)(x-x_0) + O((x-x_0)^2)\\
    V_{[\alpha]}'(x) &= V_{[\alpha]}'(x_0)+V_{[\alpha]}''(x_0)(x-x_0) + O((x-x_0)^2)\\
    V_0(x) &= V_0(x_0) + V_0'(x_0)(x-x_0) + O((x-x_0)^2)\\
    V_0'(x) &= V_0'(x_0) + V_0''(x_0)(x-x_0) + O((x-x_0)^2)\\
    V_1(x) &= V_{1}'(x_0)(x-x_0) + O((x-x_0)^2).
\end{align*}
Here $O((x-x_0)^n)$ denotes functions $f$ such that for some neighbourhood of $x_0$, there is some constant $C>0$ such that
\begin{equation*}
   \lvert f(x)\rvert \leq C\lvert x-x_0 \rvert^n.
\end{equation*}
Substituting these expansions into the definition of $[V_{[\alpha]},V_0]$ we have
\begin{align*}
    [V_{[\alpha]},V_0](x) &= V_{[\alpha]}(x)V_0'(x)-V_0(x)V_{[\alpha]}'(x) \\
    &=-V_{[\alpha]}'(x_0)V_0(x_0)-V_{[\alpha]}''(x_0)V_0(x_0)(x-x_0)+ O((x-x_0)^2).
\end{align*}
Then, expanding the left hand side and right hand side of \eqref{eq:LOACforV1}, we have
\begin{align*}
    V_{[\alpha]}(x)[V_{[\alpha]},V_0](x) & = -\lvert V_{[\alpha]}'(x_0)\rvert^2V_0(x_0)(x-x_0)\\
    &-V_{[\alpha]}'(x_0)V_{[\alpha]}''(x_0)V_0(x_0)(x-x_0)^2+ O((x-x_0)^3)\\
    \end{align*}
    and 
    \begin{align*}
    \lvert V_{[\alpha]}(x)\rvert^2 &= V_{[\alpha]}'(x_0)^2(x-x_0)^2 + O((x-x_0)^3);
\end{align*}
hence, by \eqref{eq:LOACforV1}, 
	\begin{align*}
	-\lvert V_{[\alpha]}'(x_0)\rvert^2V_0(x_0)(x-x_0)-V_{[\alpha]}'(x_0)V_{[\alpha]}''(x_0)V_0(x_0)(x-x_0)^2 & \leq -\lambda(x) V_{[\alpha]}'(x_0)^2(x-x_0)^2 \\
	&+ O((x-x_0)^3).
	\end{align*}	
	Rearranging the above gives
	\begin{equation*}
	-V_0(x_0)V_{[\alpha]}'(x_0)^2\frac{(x-x_0)}{\lvert x-x_0\rvert} \leq 	\left(V_{[\alpha]}'(x_0)V_0(x_0)V_{[\alpha]}''(x_0)-\lambda(x) V_{[\alpha]}'(x_0)^2\right)\frac{(x-x_0)^2}{\lvert x-x_0\rvert} + O((x-x_0)^2).
	\end{equation*}	
	Suppose that $V_0(x_0)V_{[\alpha]}'(x_0)\neq 0$; then letting $x$ tend to $x_0$, we obtain a contradiction. Therefore $V_0(x_0)V_{[\alpha]}'(x_0)$ must be equal to zero which implies that $[V_{[\alpha]},V_0](x_0)$ is equal to zero as well. Moreover, since $V_1(x_0)=V_{[\alpha]}(x_0)=0$, we also have $[V_{[\alpha]},V_1](x_0)=0$. Then by induction we have $V_{[\alpha]}(x_0)=0$ for all $\alpha\in \A_m$. This concludes the proof.
\end{proof}

\begin{lemma}\label{lem:hypcheckforadditivenoiseex}
	Consider the SDE \eqref{eq:SDEadditivenoise1dim}. If there exists $R>0$ such that  $\mathrm{sign}(x)b(x)<0$ whenever $\lvert x\rvert\geq R$ then Hypothesis \ref{hyp:upperbound} holds with the sequence $u_n(x)$ as in \eqref{unlemb6} below and 
	\begin{equation*}
	    \DV(x) = \alpha^2+\alpha b(x)\tanh(\alpha x) \, ,
	\end{equation*}
	where $\alpha>0$ is any positive constant. 
	Moreover, if $b$ is unbounded (both above and below) then Hypothesis \ref{hyp:Donskerhyp} also holds.
\end{lemma}

\begin{proof}[Proof of Lemma \ref{lem:hypcheckforadditivenoiseex}]
	Let $u(x)=\cosh(\alpha x)$ and define $\theta:\R\to\R$ as in Note \ref{rem:Donskerhyp}. 
	Then Hypothesis \ref{hyp:upperbound} is satisfied with the functions
	\begin{align}
	u_n(x) &= u\left(n\theta\left(\frac{x}{n}\right)\right), \label{unlemb6}\\
	\DV(x) &= \alpha b(x)\tanh(\alpha x)+\alpha^2,\nonumber
	\end{align}
	as we come to explain.
	 By construction $u_n(x) \leq u(x)$ for all $x\in\R$ and $n\in\N$, so for each compact set $W\subseteq \R$
	\begin{equation*}
	\sup_{x\in W} \sup_{n\in\N} u_n(x) \leq \sup_{x\in W}  u(x) <\infty,
	\end{equation*}
	so Hypothesis \ref{hyp:Donskerhyp} \ref{ass:upperboundforun}.
	Now for fixed $x\in\R$ and $n >\lvert x\rvert$ we have $u_n(x)=u(x)$ and
	\begin{equation*}
	\frac{\cL u(x)}{u(x)} = \alpha b(x)\tanh(\alpha x)+\alpha^2=\DV(x),
	\end{equation*}
	therefore Hypothesis \ref{hyp:Donskerhyp} \ref{ass:limitdefofV} is satisfied. Moreover we see that if $\lvert x\rvert \leq n$ then 
	\begin{equation*}
	\frac{\cL u_n}{u_n} \leq \DV(x) \leq A_1,
	\end{equation*}
	where $A_1$ is the maximum value of $\DV$. Now if $\rvert x\lvert\geq 2n$ then $u_n(x)$ is constant and we have
	\begin{equation*}
	\frac{\cL u_n}{u_n} =0.
	\end{equation*}	
	However if $n\leq \lvert x\rvert \leq 2n$ then
	\begin{equation*}
	\frac{\cL u_n(x)}{u_n(x)} = \left(\alpha b(x)\theta'\left(\frac{x}{n}\right) + \alpha^2\frac{1}{n}\theta''\left(\frac{x}{n}\right)\right)\tanh\left(\alpha n\theta\left(\frac{x}{n}\right)\right) + \alpha^2\theta'\left(\frac{x}{n}\right).
	\end{equation*}
	We may assume that $n$ is sufficiently large that $b(x)\tanh(\alpha n\theta(\frac{x}{n}))<0$. Then since $\theta$ is increasing on $[1,2]$ and is an odd function we have that $b(x)\theta'(x/n)\tanh(\alpha n\theta(x/n))<0$ thus 
	\begin{equation*}
	\frac{\cL u_n(x)}{u_n(x)}\leq   \alpha^2\frac{1}{n}\theta''\left(\frac{x}{n}\right)\tanh\left(\alpha n\theta\left(\frac{x}{n}\right)\right) + \alpha^2\theta'\left(\frac{x}{n}\right) \leq \alpha^2\sup_{y\in[1,2]} [\lvert\theta''(y)\rvert + \theta'(y)]=:A_2.
	\end{equation*}
	Therefore \eqref{eq:boundsforLun} holds with $A=\max\{A_1,A_2\}$.
	
	Moreover, if $b$ is unbounded (both above and below) and $\mathrm{sign}(x)b(x)<0$ for $x$ sufficiently large, we see that the set $\{x\in\R:\DV(x)\geq \ell\}$ is compact for each $\ell\in\R$.
\end{proof}

\begin{proof}[Proof of Theorem \ref{prop:LOAChd}]	
Analogously to the one-dimensional setting we shall denote by $\J_t=\J_t^x=\frac{\partial}{\partial x}X_t^{(x)}$ the $N\times N$ matrix valued process which denotes the derivative of $X_t^{(x)}$ with respect to $x$; this exists by \cite[Theorem 7.3]{Kunita} and can be viewed as the solution of
\begin{equation}\label{eq:SDEforJacobianhd}
d\J_t^{x}= \left(\jacobian{U_0}{x}\right)(\Xx)\J_t^{x} dt + \sqrt{2}\sum_{i=1}^d\left(\jacobian{V_i}{x}\right)(\Xx)\J_t^{x} dB_t^i, \quad \J_0^{x}=\mathrm{Id}\,.
\end{equation}	
With this notation in place, we  rewrite derivatives of the semigroup in terms of derivatives of the process $X_t^{(x)}$. By a completely analogous argument to the proof of Lemma \ref{lem:derivativeofsemigp} we have
	\begin{align}
	V\cP_tf(x) = 
	\mathbb{E}[\nabla f(\Xx)^T \J_t V(x)]\label{eq:devofsemigpintermsofVhd}
	\end{align}
	for every $x\in \R^N$ and $f\in \CV$. For clarity we emphasize that here $\nabla f(\Xx)$ denotes the  gradient of $f$ evaluated at $\Xx$ and that on the LHS of \eqref{eq:devofsemigpintermsofVhd} $V$ is intended as a differential operator while on the RHS we view it as a vector field.  Let us introduce the two parameter random process $\{\Gamma_{s,t}^V\}_{0\leq s\leq t}$, defined as follows:
\begin{equation*}
\Gamma_{s,t}^V = \left\lvert \nabla f(\Xx)^T\J_t\J_s^{-1}V(X_s^{(x)})\right\rvert^2.
\end{equation*}
Notice that by \eqref{eq:devofsemigpintermsofVhd} we have
\begin{align*}
\lvert V\cP_tf(x) \rvert^2&\leq \mathbb{E}\left[\left\lvert  \nabla f(\Xx)^T\J_tV(x)\right\rvert^2\right]= \mathbb{E}[\Gamma_{0,t}^V],
\end{align*}
and moreover, (using that $f$ belongs to $\CV$) we may estimate $\Gamma_{t,t}^V$ by
\begin{equation*}
\Gamma_{t,t}^V =  \lvert Vf(X_t^{(x)})\rvert^2\leq \lVert Vf\rVert_\infty^2.
\end{equation*}	 
Hence to prove \eqref{eq:gradesthd} it is sufficient to prove the following inequality
\begin{equation}\label{eq:requiredestimatehd}
\mathbb{E}[\Gamma_{0,t}^V]\geq  \mathbb{E}\left[\exp\left(-2\int_0^t \lambda(X_s^{(x)})ds \right)\Gamma_{t,t}^V\right].
\end{equation}
We will use \cite[Equation (2.63)]{Nualart} which, in our notation and setting, can be written as
	\begin{equation}\label{eq:jacinvhd}
	d(J_t^{-1}V(\Xx)) = J_t^{-1}[V_0,V](X_t^{(x)})dt+\sqrt{2}\sum_{k=1}^dJ_t^{-1}[V_k,V](X_t^{x})\circ dB_t^k.
	\end{equation}
Because of our commutativity assumption, the commutator in front of the noise in \eqref{eq:jacinvhd}  disappears; with this in mind we obtain
	\begin{align*}
	d\left(\J_t^{-1}V(\Xx)V(\Xx)^T(J_t^{-1})^T\right)&=   \J_t^{-1}[V_0,V](\Xx)V(\Xx)^T(\J_t^{-1})^Tdt\\
	&+ \J_t^{-1}V(\Xx)[V_0,V](\Xx)^T(\J_t^{-1})^Tdt.
	\end{align*}
	Integrating from $0$ to $s$, multiplying by $\nabla f(\Xx)^T \J_t$ on the left and $\J_t^T\nabla f(\Xx)$ on the right one gets
	\begin{align*}
	&\left\lvert \nabla f(\Xx)^T\J_t\J_s^{-1} V(\Xx) \right\rvert^2 =\left\lvert \nabla f(\Xx)^T\J_t V(x)\right\rvert^2\\
	&+ 2\int_0^s \nabla f(\Xx)^T \J_t\J_r^{-1}[V_0,V](X_r^{(x)})V(X_r^{(x)})^T(\J_r^{-1}\J_t)^T \nabla f(\Xx)dr.
	\end{align*}
	As in the one dimensional setting we may define $f_{s,t}=f\circ \Phi_{s,t}$, so that $(\nabla f_{s,t} (X_s^{(x)}))^T= (\nabla f(X_t^{(x)}))^T J_t J_s^{-1}$,  and we have
	\begin{align*}
	\left\lvert  Vf_{s,t}(X_s^{(x)}) \right\rvert^2 &=\left\lvert Vf_{0,t}(x)\right\rvert^2+ 2\int_0^s [V_0,V]f_{r,t}(X_r^{(x)})Vf_{r,t}(X_r^{(x)})dr.
	\end{align*}
	Now we may apply \eqref{LOAC1d} and obtain
	\begin{align*}
	\left\lvert  Vf_{s,t}(X_s^{(x)}) \right\rvert^2 &\geq\left\lvert Vf_{0,t}(x)\right\rvert^2+ 2\int_0^s\lambda(X_t^{(x)}) \lvert Vf_{r,t}(X_r^{(x)})\rvert^2dr.
	\end{align*}
	We can rewrite this in terms of $\Gamma_{s,t}^V$ as
	\begin{align*}
	\Gamma_{s,t}^V &\geq \Gamma_{0,t}^V+ 2\int_0^s\lambda(X_r^{(x)}) \Gamma_{r,t}^V dr.
	\end{align*}
	That is,
	\begin{equation}\label{eq:Gammaineqhd}
	\exp\left(-2\int_0^s(\lambda(X_r^{(x)}))  dr \right)\Gamma_{s,t}^V \geq \Gamma_{0,t}^V. 
	\end{equation}
	Taking expectations and setting $s=t$ one obtains \eqref{eq:requiredestimatehd}. This concludes the proof.
\end{proof}

\begin{lemma}\label{lem:expmomentsforarctan}
    Let $X_t^{(x)}$ be the solution of the one-dimensional SDE \eqref{eq:arctanSDE}. Then
    \begin{equation}\label{eq:expmomentsofarctansde}
        \sup_{t\geq 0}\mathbb{E}[\cosh( X_t^x)] <\infty.
    \end{equation}
    Moreover, let $\{Y_{t_n}^\delta\}_{n\in\N}$ denote the Euler approximation of $\{X_t^{(x)}\}$ with initial condition $x$; then there exists some $\delta^\ast>0$ such that for all $\delta\in (0,\delta^\ast)$ we have
    \begin{equation*}
    \sup_{n\in \N} \mathbb{E}[\cosh( Y_{t_n}^\delta)]<\infty.
    \end{equation*}
\end{lemma}

\begin{proof}[Proof of Lemma \ref{lem:expmomentsforarctan}]
    Start by observing that for any $\alpha\in (1,\pi/2)$ we may find $\beta\in\R$ such that the following inequality holds
    \be\label{eq:Lyapunovcondforarctan}
    -\arctan(x)\sinh(x) \leq \beta-\alpha\cosh(x), \text{ for every } x\in\R.
    \ee
    This implies that $\cosh(x)$ is a Lyapunov function for the SDE \eqref{eq:arctanSDE} or, more precisely, we have 
    \begin{equation*}
    \cL \cosh(x) = -\arctan(x)\sinh(x) + \cosh(x) \leq \beta-(\alpha-1)\cosh(x).
    \end{equation*}
    Then by \cite[Theorem 2.1]{Meyn} this implies that \eqref{eq:expmomentsofarctansde} holds.  
    Now by \eqref{eq:eulerdiscrete} we have
    \begin{align*}
        \mathbb{E}[\cosh(Y_{t_{n+1}}^\delta)] &= \mathbb{E}[\cosh(Y_{t_{n}}^\delta-\arctan(Y_{t_n}^\delta)\delta + \sqrt{2}\Delta B_{t_n})]\\
        &=\mathbb{E}[\cosh(Y_{t_{n}}^\delta-\arctan(Y_{t_n}^\delta)\delta)\cosh(\sqrt{2}\Delta B_{t_n})] \\&+ \mathbb{E}[\sinh(Y_{t_{n}}^\delta-\arctan(Y_{t_n}^\delta)\delta)\sinh(\sqrt{2}\Delta B_{t_n})]. 
    \end{align*}
    Now using the independence of $Y_{t_n}^\delta$ and $\Delta B_{t_n}$ and the fact that $\Delta B_{t_n}$ is distributed according to a Gaussian random variable with mean zero and variance $\delta$ we have
       \begin{equation}\label{eq:expofcosh}
          \mathbb{E}[\cosh(Y_{t_{n+1}}^\delta)]=e^\delta\mathbb{E}\left[\cosh\left(Y_{t_{n}}^\delta-\arctan(Y_{t_n}^\delta)\delta \right)\right].   
       \end{equation}
    By Taylor's theorem we know that there exists some $\theta=\theta(Y_{t_n}^\delta, \delta)\in (0,1)$ such that
    \begin{align}\label{eq:coshtaylorexp}
    \cosh\left(Y_{t_{n}}^\delta-\arctan(Y_{t_n}^\delta)\delta \right) &= \cosh(Y_{t_n}^\delta) -\arctan(Y_{t_n}^\delta)\sinh(Y_{t_n}^\delta)\delta \nonumber\\
    &+ \frac{1}{2}\arctan(Y_{t_n}^\delta)^2 \delta^2 \cosh(Y_{t_n}^\delta - \theta \delta \arctan(Y_{t_n}^\delta)).
    \end{align}
    Note that if $\lvert Y_{t_n}^\delta \rvert >\pi/2$ then $\cosh(Y_{t_n}^\delta - \theta \delta \arctan(Y_{t_n}^\delta))\leq \cosh(Y_{t_n}^\delta)$ otherwise $\cosh(Y_{t_n}^\delta - \theta \delta \arctan(Y_{t_n}^\delta))\leq \cosh(\pi \delta/2)$, so overall 
    \begin{equation}\label{eq:coshdetestimate}
    \cosh(Y_{t_n}^\delta - \theta \delta \arctan(Y_{t_n}^\delta))\leq \cosh(\pi \delta/2)+\cosh(Y_{t_n}^\delta).
    \end{equation}
    Hence we can bound the right hand side of \eqref{eq:coshtaylorexp} from above using that \eqref{eq:coshdetestimate} and \eqref{eq:Lyapunovcondforarctan}, to obtain 
      \begin{equation*}
    \cosh\left(Y_{t_{n}}^\delta-\arctan(Y_{t_n}^\delta)\delta \right) \leq  \cosh(Y_{t_n}^\delta) +(\beta-\alpha\cosh(Y_{t_n}^\delta))\delta + \frac{\pi^2}{8} \delta^2 \cosh(Y_{t_n}^\delta)+ \frac{\pi^2}{8} \delta^2 \cosh\left(\frac{\pi}{2} \delta\right).
    \end{equation*}
    Substituting this inequality into \eqref{eq:expofcosh} we have
    \begin{equation*}
   \mathbb{E}[\cosh(Y_{t_{n+1}}^\delta)]\leq e^\delta(1-\alpha\delta+\frac{\pi^2}{8}\delta^2)\mathbb{E}[\cosh(Y_{t_n}^\delta)] +e^\delta\beta\delta+ e^\delta\frac{\pi^2}{8} \delta^2 \cosh\left(\frac{\pi}{2} \delta\right).     
    \end{equation*}
    By recursion on the right hand side of the above inequality we get
    \begin{equation*}
    \mathbb{E}[\cosh(Y_{t_{n}}^\delta)]\leq \mathbb{E}[\cosh(Y_{0}^\delta)] a^n + \frac{b(1-a^n)}{1-a}
    \end{equation*}
    where $a=e^\delta(1-\alpha\delta+\frac{\pi^2}{8}\delta^2)$ and $b=e^\delta\beta\delta+ e^\delta\frac{\pi^2}{8} \delta^2 \cosh\left(\frac{\pi}{2} \delta\right)$. Now note that for $\delta$ sufficiently small we have that $0<a<1$, in which case 
    \begin{equation*}
    \mathbb{E}[\cosh(Y_{t_{n}}^\delta)]\leq \mathbb{E}[\cosh(Y_{0}^\delta)] + \frac{b}{1-a}.
    \end{equation*}
    This gives the required estimate since $a$ and $b$ do not depend on $n$.
\end{proof}

\begin{proof}[Proof of Lemma \ref{lem:higherorderOACforadditivenoise}]
    We shall only prove \eqref{eq:expdecayuptoorder4nonuniform} for $\partial_x^4\cP_tf(x)$ as the other estimates follow by a simpler version of the same argument. Recall $J_t$ is defined as $J_t=\frac{\partial}{\partial x}X_t^{(x)}$ and we can similarly define higher order derivatives as $J_t^{(n)} = \frac{\partial^n}{\partial x^n}X_t^{(x)}$.
    
    Using the chain rule and then taking expectations, similiarly to Lemma \ref{lem:derivativeofsemigp}, we obtain
    \begin{align}
    \partial_x^4\cP_tf(x)  &=\mathbb{E}\bigg[ f^{(4)}(X_t^{(x)})J_t^4 + 6f^{(3)}(X_t^{(x)})J_t^2J_t^{(2)} \\
    &+3f''(X_t^{(x)})(J_t^{(2)})^2+4f''(X_t^{(x)})J_tJ_t^{(3)}+ f'(X_t^{(x)})J_t^{(4)}\bigg]\\&
    \leq K\lVert f\rVert_{C_b^4(\R)} \mathbb{E}\left[ J_t^4 + J_t^2\lvert J_t^{(2)}\rvert +(J_t^{(2)})^2+J_t\lvert J_t^{(3)}\rvert+ \lvert J_t^{(4)}\rvert\right].\label{eq:expansionof4derofsemigp}
    \end{align}
    Now we proceed by estimating each of these terms in turn. Note that in the case at hand \eqref{eq:SDEforJacobian} simplifies to
    \begin{equation}\label{eq:jacODE}
        dJ_t = b'(X_t^{(x)})dt
    \end{equation}
    which we can solve to find
     \begin{equation}\label{eq:explicitJ}
        J_t = \exp\left(\int_0^t b'(X_s^{(x)})ds\right).
    \end{equation}
Differentiating \eqref{eq:explicitJ} we obtain the following expressions for the higher order derivatives,
    \begin{align*}
        J_t^{(2)} &= J_t \int_0^t b''(X_s^{(x)})J_s ds \\
        J_t^{(3)} &= \left(\int_0^t b^{(3)}(X_s^{(x)}) J_s^2 ds+\int_0^t b''(X_s^{(x)})J_s^{(2)}ds \right)J_t\\
        &+ J_t^{(2)}\int_0^t b''(X_s^{(x)})J_sds \\
        J_t^{(4)} &=\left(\int_0^t b^{(4)}(X_s^{(x)})J_s^3 +3b^{(3)}(X_s^{(x)})J_sJ_s^{(2)} + b''(X_s^{(x)})J_s^{(3)} ds\right)J_t\\
        &+\left(\int_0^t 3b^{(3)}(X_s^{(x)})J_s^2 +2b''(X_s^{(x)})J_s^{(2)}ds\right)J_t^{(2)}\\
        &+ J_t^{(3)}\int_0^t b''(X_s^{(x)})J_s ds.
    \end{align*}
    From here it is straight forward to see that the conclusion holds.

    \end{proof}

\end{document}